\documentclass[a4paper,10pt]{amsart}
\usepackage[arrow,matrix]{xy}
\usepackage{amsmath,amssymb,amscd,bbm,amsthm,mathrsfs,dsfont,bm}
\usepackage{extarrows}
\usepackage{latexsym}
\usepackage{graphicx}
\usepackage{color}
\allowdisplaybreaks
\newtheorem{theorem}{Theorem}[section]
\newtheorem{lemma}[theorem]{Lemma}
\newtheorem{proposition}[theorem]{Proposition}
\newtheorem{corollary}[theorem]{Corollary}
\theoremstyle{definition}
\newtheorem{definition}[theorem]{Definition}
\newtheorem{question}[theorem]{Question}
\newtheorem{remark}[theorem]{Remark}

\numberwithin{equation}{theorem}

 \DeclareMathOperator{\Ker}{Ker}
\DeclareMathOperator{\Image}{Im} \DeclareMathOperator{\Ext}{Ext}
 
\DeclareMathOperator{\Hom}{Hom} \DeclareMathOperator{\uExt}{\underline{Ext}} \DeclareMathOperator{\uHom}{\underline{Hom}}
\def\bt{\begin{theorem}}
\def\et{\end{theorem}}
\def\bl{\begin{lemma}}
\def\el{\end{lemma}}
\def\br{\begin{remark}}
\def\er{\end{remark}}
\def\bc{\begin{corollary}}
\def\ec{\end{corollary}}

\begin{document}
\title [Regularity criterion and classification for algebras of Jordan type]
{Regularity criterion and classification\\ for algebras of Jordan type
}
\author{Y. Shen, G.-S. Zhou and D.-M. Lu}
\address{Department of Mathematics, Zhejiang University,
Hangzhou 310027, China}
\email{11206007@zju.edu.cn;\quad 10906045@zju.edu.cn;\quad dmlu@zju.edu.cn}
\date{}

\begin{abstract} We show that Artin-Schelter regularity of a $\mathbb{Z}$-graded algebra can be examined by its associated $\mathbb{Z}^r$-graded algebra. We prove that there is exactly one class of four-dimensional Artin-Schelter regular algebras with two generators of degree one in the Jordan case. This class is strongly noetherian, Auslander regular, and Cohen-Macaulay. Their automorphisms and point modules are described.
\end{abstract}




\maketitle

\section*{Introduction}
The classification of Artin-Schelter regular algebras, or classification of quantum projective spaces, is one of important questions in noncommutative projective algebraic geometry. Many researchers have been interested in Artin-Schelter regular
algebras and many have made great contributions on the subject.
In the case of global dimension 4, plenty of Artin-Schelter regular algebras have been
discovered in recent years \cite{LPWZ2,RZ,ZZ1,ZZ2}, and
most of them are obtained by endowing with an appropriate
$\mathbb{Z}^2$-grading on the algebras. It is not the case for 4-dimensional Artin-Schelter regular algebras of Jordan type (see the subsection 3.1 below). This motivates us to study this kind of algebras.

The idea used here is to link a $\mathbb{Z}$-graded algebra with an appropriate $\mathbb{Z}^r$-graded algebra for some positive integer $r (>1)$. By means of Gr\"obner basis theory, we show it is available for those $\mathbb{Z}$-graded algebras without $\mathbb{Z}^2$-grading on them. Our first result is a regularity criterion for a connected graded algebra.

\begin{theorem}\label{first theo}
Let $A=k\langle X\rangle/I\/$ be a connected graded algebra. Then $A$ is Artin-Schelter regular in case there is an appropriate $\mathbb{Z}^r$-grading on $k\langle X\rangle$ such that  $\mathbb{Z}^r$-graded algebra $k\langle X\rangle/(LH(\mathcal{G}))$ is Artin-Schelter regular, where $\mathcal{G}$ is the reduced Gr\"obner basis of $I$ with respect to an admissible order $\prec_{\mathbb{Z}^r}$ on $X^*$.
\end{theorem}

An application of the criterion in this paper is the connected Artin-Schelter regular algebras of dimension 4 with two generators whose Frobenius data is in Jordan case. The generic constraints condition (see \cite{LPWZ2}) in this case turns out to be invalid. Using the $A_\infty$-algebra theory and applying the criterion to the case, we get a classification result:

\begin{theorem}
The algebra $\mathcal{J}=\mathcal{J}(u, v, w)=k\langle x,y\rangle/(f_1,f_2)$ is an Artin-Schelter regular algebra of global dimension $4$, where
\begin{eqnarray*}
\begin{aligned}
f_1&=xy^2-2yxy+y^2x,\\
f_2&=x^3y-3x^2yx+3xyx^2-yx^3+(1-u)xyxy+uyx^2y\\
&\hskip10mm +(u-3)yxyx+(2-u)y^2x^2-vy^2xy+vy^3x+wy^4,
\end{aligned}
\end{eqnarray*}
and $u, v, w\in k$.

If $k$ is algebraically closed of characteristic $0$, then it is, up to isomorphism, the unique Artin-Schelter regular algebra of global dimension $4$ which is generated by two elements whose Frobenius data is the Jordan type.
\end{theorem}

As the criterion for Artin-Schelter regularity, we provide a similar method to recognize the ring-theoretic and homological properties of an Artin-Schelter regular algebra from the known one.
\begin{theorem}
Let $\mathcal{J}$ be the Artin-Schelter regular algebra showed in the theorem above. Then
\begin{enumerate}
\item $\mathcal{J}$ is strongly noetherian, Auslander regular and Cohen-Macaulay;
\item The automorphism group of $\mathcal{J}$ is
 $\left\{
{\small\left(
\begin{array}{cc}
 a&b\\
0&a
\end{array}
\right)}\;
\Big|\;\; a\in k\backslash\{0\},\, b\in k
\right\}$;
\item $\mathcal{J}$ has two classes of point modules up to isomorphism.
\end{enumerate}
\end{theorem}

Here is an outline of the paper. In Section 1, we review some basic definitions of Artin-Schelter regular algebras, $A_\infty$-algebras and $\mathbb{Z}^r$-filtered algebras. The links about regularity between $\mathbb{Z}^r$-filtered algebras and associated $\mathbb{Z}^r$-graded algebras are considered in Section 2. The next two sections are devoted to an application of the criterion to the classification of Artin-Schelter regular algebras of Jordan type. Properties of the classified result are presented in Section 5.

Throughout the paper, $k$ is a fixed algebraically closed field of characteristic $0$. The set of natural numbers $\mathbb{N}=\{0, 1, 2,\cdots\}$. Unless otherwise stated, graded means $\mathbb{Z}$-graded, the tensor product $\otimes$ means $\otimes_k$. For simplicity we only consider graded algebras that are generated in degree 1. We also set $\mathbb{Z}^r=\underbrace{\mathbb{Z}\times\cdots\times\mathbb{Z}}_r$ with the standard basis $\varepsilon_i=(0,\cdots,1,\cdots,0)$ for $i=1, 2,\cdots, r$.

\section{Preliminaries}
In this section we recall the definitions of Artin-Schelter regular algebras, $A_\infty$-algebras, and $\mathbb{Z}^r$-filtered algebras as well as some fundamental consequents in preparation for the classification.

A \emph{norm map} on $\mathbb{Z}^r$ we mean the map $\|\cdot\|: \mathbb{Z}^r\to \mathbb{Z}$ which sends $\alpha=(a_1, \cdots, a_r)$ to $\sum_{i=1}^ra_i$. An \emph{admissible ordering} $<$ related to the norm map is a total ordering such that $\|\alpha_1\|<\|\alpha_2\|$ implies $\alpha_1<\alpha_2$, and $\alpha_1<\alpha_2$ implies $\alpha_1+\alpha_3<\alpha_2+\alpha_3$ for any $\alpha_1,\alpha_2,\alpha_3\in\mathbb{Z}^r$.

Let $A=\bigoplus_{\alpha\in \mathbb{Z}^r}A_\alpha$ be a $\mathbb{Z}^r$-graded algebra. For a homogenous element $a\in A_\alpha$, we call $\alpha$ and $\|\alpha\|$ the \emph{degree} and \emph{total degree} of $a$, denoted by $\deg a$ and $\mathrm{tdeg\,}a$, respectively. $A$ is called \emph{connected} if $A_\alpha=0$ for all $\alpha\notin \mathbb{N}^r$ and $A_0=k$. A connected $\mathbb{Z}^r$-graded algebra $A$ is called \emph{properly} if $A$ is generated by $\bigoplus_{i=1}^rA_{\varepsilon_i}$. We denote by $\mathrm{GrMod\,}A$ the category of $\mathbb{Z}^r$-graded left $A$-modules with morphisms of $A$-homomorphisms preserving degrees, and by $\mathrm{grmod\,}A$ the category of finitely generated $\mathbb{Z}^r$-graded left $A$-modules. The categories of $\mathbb{Z}^r$-graded right $A$-modules, denoted by $\mathrm{GrMod\,}A^o$ and $\mathrm{grmod\,}A^o$ respectively, are defined analogously. When $r=1$, it goes back to the usual graded situation.

Given a $\mathbb{Z}^r$-graded $A$-module $M=\bigoplus_{\alpha\in \mathbb{Z}^r}M_\alpha$ and $\beta\in \mathbb{Z}^r$, its \emph{shift}  is $M(\beta)\in\mathrm{GrMod\,}A$ defined by $M(\beta)_\alpha=M_{\alpha+\beta}$ for any $\alpha\in\mathbb{Z}^r$.
For $M, N\in\mathrm{GrMod\,}A$, we write $\uHom_A(M,N)=\bigoplus_{\alpha\in\mathbb{Z}^r}\Hom_{\text{GrMod}\,A}(M, N(\alpha))$ and $\uExt^i_A(M,N)=\bigoplus_{\alpha\in\mathbb{Z}^r}\Ext^i_{\text{GrMod}\,A}(M, N(\alpha))$.

\medskip
\subsection{Artin-Schelter regular algebras}~

The following definition is originally due to Artin and Schelter \cite{AS}.
\begin{definition}
A connected $\mathbb{Z}^r$-graded algebra $A$ is called \emph{Artin-Schelter regular} (AS-regular, for short) of dimension $d$ if the following three conditions hold:
\begin{enumerate}
\item[(AS1)] $A$ has finite global dimension $d$;
\item[(AS2)] $A$ has finite Gelfand-Kirillov dimension;
\item[(AS3)] $A$ is Gorenstein; that is, for some $l\in\mathbb{Z}^r$,
\begin{eqnarray*}
\uExt^i_A(k,A)=
\begin{cases}
k(l)&\text{if}~i=d,\cr
0&\text{if}~i\neq d,
\end{cases}
\end{eqnarray*}
where $l$ is called \emph{Gorenstein parameter}.
\end{enumerate}
\end{definition}

The following proposition was originally proved for $\mathbb{Z}$-graded algebras, and holds true in our $\mathbb{Z}^r$-setting.

\begin{proposition}\cite[Proposition 3.1]{SZ}
Let $A$ be a $\mathbb{Z}^r$-graded AS-regular  algebra of dimension $d$ with Gorenstein parameter $l$, then the minimal projective resolution of $_Ak$ is
$$
0\rightarrow P_d\rightarrow  P_{d-1}\rightarrow\cdots\rightarrow
P_1\rightarrow A\rightarrow _Ak\rightarrow 0.
$$
where $P_j=\bigoplus_{i=1}^{s_{j}}Ae_i^j$ is a finitely generated free module on basis $\{e_i^j\}_{i=1}^{s_{j}}$ with $\deg e_i^j\in\mathbb{N}^r$ for all $1\leq j\leq d-1$, and $P_d=Ae_d$ is a free module on $e_d$ with $\deg e_d=l$.
\end{proposition}

Artin and Schelter in \cite{AS} conjecture that all AS-regular algebras are noetherian. The examples of AS-regular algebras found so far are in their guess. It seems a fundamental assumption for AS-regular algebras. With this assumption, some abstract properties have been proved in small dimensions. For example, any noetherian connected graded AS-regular algebra of global dimension $4$ and GK$\dim A=4$ is a domain (see \cite[Theorem 3.9]{ATV2}). Hence, in the following we assume that AS-regular algebras are domains.

As the paper \cite{LPWZ2} observes, the graded AS-regular algebras of global dimension 4 which are domains have three resolution types as they named $(14641)$, $(13431)$ and $(12221)$ according to the number of generators in degree 1.

In Sections 3--5, we will focus on the connected graded algebras $A$ of type $(12221)$ whose minimal resolution of trivial module $_Ak$ is
\begin{equation*}\label{resolution}
0\rightarrow A(-7)\rightarrow A(-6)^{\oplus2}\rightarrow A(-4)\oplus A(-3)\rightarrow A(-1)^{\oplus2}\rightarrow A\rightarrow _Ak\rightarrow 0\tag{$*$}
\end{equation*}
and the Hilbert series is
\begin{equation*}
H_A(t)=\frac{1}{(1-t)^3(1-t^2)(1-t^3)}.
\end{equation*}

In some literatures, the condition $\mathrm{(AS2)}$ is not required in the definition of Artin-Schelter regular algebra. The following theorem is a pivotal point in the classification which was proved by using $A_\infty$-algebra method.
\begin{theorem}\label{ASFrobenius}
$($\cite[Theorem 12.5]{LPWZ1}, \cite[Corollary D]{LPWZ4}$)$
Let $A$ be a connected graded algebra, and let $E(A):=\Ext^*_A(k, k)$ be the Yoneda algebra of $A$. Then $A$ satisfies $\mathrm{(AS1)}$ and $\mathrm{(AS3)}$ if and only if $E(A)$ is Frobenius.
\end{theorem}

\medskip
\subsection{$A_\infty$-algebras}~

The definition and notation of the $A_\infty$-algebra are introduced  in this subsection briefly.  We refer to \cite{Ke1, LPWZ2} for the details.
\begin{definition}
Let $E=\bigoplus_{i\in \mathbb{Z}} E^i$ be a $\mathbb Z$-graded $k$-vector space. $E$ is an \emph{$A_\infty$-algebra} if it is endowed with a family of graded $k$-linear maps
 $$
 m_n:~E^{\otimes n}\rightarrow E,\quad n\geq 1,
 $$
of degree $2-n$ satisfying the following \emph{Stasheff identities} $\textsf{SI(n)}$:
$$
\sum_{\substack{r+s+t=n;\\s\geq 1;\ r,t\geq 0}}(-1)^{r+st}m_{r+1+t}(id^{\otimes r}\otimes m_s\otimes id^{\otimes t})=0
$$
for all $n\geq 1$.
\end{definition}

Note that when the formulas are applied to elements, additional signs appear due to the Koszul sign rule. We assume every $A_\infty$-algebra in this article contains an identity element $1\in E^0$ which satisfies \emph{strictly unital condition}; that is,
\begin{enumerate}
\item $m_2(1, x)=x$ and $m_2(x, 1)=x$, for every $x\in E$;
\item $m_n(x_1, \cdots, x_n)=0$, if $x_i=1$ for some $i$ and $n\neq 2$.
\end{enumerate}
\vspace{2mm}

Now let $A$ be a connected graded algebra, then the \emph{Yoneda algebra} $E(A)$ is bi-graded naturally: one is the \emph{homological degree}, written as superscript, and the other one is \emph{Adams degree}, written as subscript, the latter is induced by the grading of $A$. It is a basic fact that $E(A)$ can be viewed as the cohomology algebra of some differential graded algebra.

For any differential graded algebra $D$, there is a canonical $A_\infty$-algebra structure on its cohomology algebra $HD$ which is unique in the sense of $A_\infty$-isomorphisms. This is a key consequent in $A_\infty$-world named ``\emph{Minimal Model Theorem}'' (see \cite{Ke1}). A concrete method of constructing the minimal model is provided in \cite{Me}. As a consequence, $E(A)$ is equipped with a natural $A_\infty$-algebra structure.
We adopt that the $A_\infty$-algebra structure in this paper are bi-graded and all multiplications $\{m_n\}$ preserve Adams degree; that is, $\deg(m_n)=(2-n, 0)$. It is a nontrivial hypothesis since such an $A_\infty$-algebra structure exists (see \cite{LPWZ3}).

We use $E(A)$ to denote both the usual associative Yoneda algebra
and the $A_\infty$-Ext-algebra with any choice of its $A_\infty$-structure. There is such a graded algebra $A$ that its associative Yoneda algebra $E(A)$ does
not contain enough information to recover the original algebra $A$; on the other
hand, the information from the $A_\infty$-algebra $E(A)$ is sufficient to recover $A$. This is the point of the following theorem:

\begin{theorem}\label{recover}
$($\cite[Corollary B]{LPWZ3}$)$ Let $A$ be a connected graded algebra which is finitely generated in degree $1$ and $E$ be the $A_\infty$-algebra $\mathrm{Ext}_A^*(k_A,k_A)$. Let $R=\bigoplus_{n\geq2}R_n$ be a minimal graded space of relations of $A$ such that $R_n\subset A_1\otimes A_{n-1}\subset A_1^{\otimes n}$. Let $i:R_n\rightarrow A_1^{\otimes n}$ be the inclusion map, and $i^\sharp$ be its k-linear dual. Then the multiplication $m_n$ of $E$ restricted to $(E^1)^{\otimes n}$ is equal to the map
$$
i^\sharp:(E^1)^{\otimes n}=(A_1^\sharp)^{\otimes n}\rightarrow R^\sharp_n\subset E^2.
$$
where $R^\sharp$ and $A_1^\sharp$ is the graded $k$-linear dual of $R$ and $A_1$ respectively.
\end{theorem}

\medskip
\subsection{$\mathbb{Z}^r$-filtered algebras and modules}~

The basic definitions and notations of $\mathbb{Z}^r$-filtered algebras and $\mathbb{Z}^r$-filtered modules are given in this subsection.  We refer to \cite{LHS} for the details.

Let $<$ be a fixed admissible ordering on $\mathbb{Z}^r$.
\begin{definition}
A $k$-algebra $B$ is called a \emph{$\mathbb{Z}^r$-filtered algebra} if there is a family $\{F_\alpha(B)\}_{\alpha\in\mathbb{Z}^r}$ of $k$-subspaces of $B$ such that
\begin{enumerate}
\item $F_\alpha(B)\subseteq F_{\alpha'}(B)$ if $\alpha<\alpha'$;
\item $F_\alpha(B)F_{\alpha'}(B)\subseteq F_{\alpha+\alpha'}(B)$, for any $\alpha,\alpha'\in \mathbb{Z}^r$;
\item $B=\bigcup_{\alpha\in\mathbb{Z}^r} F_\alpha(B)$, and $1\in F_0(B)$.
\end{enumerate}
In the definition above, the family $\{F_\alpha(B)\}_{\alpha\in\mathbb{Z}^r}$ is called a \emph{$\mathbb{Z}^r$-filtration} of $B$.
\end{definition}

An \emph{associated $\mathbb{Z}^r$-graded algebra} of $\mathbb{Z}^r$-filtered algebra $B$ is defined by
$$
G^r(B)=\bigoplus_{\alpha\in\mathbb{Z}^r}\frac{F_\alpha(B)}{F_{<\alpha}(B)},
$$
where $F_{<\alpha}(B)=\bigcup_{\alpha'<\alpha}F_{\alpha'}(B)$.

\begin{definition}
Let $B$ be a $\mathbb{Z}^r$-filtered algebra and $M$ a $B$-module. $M$ is called \emph{$\mathbb{Z}^r$-filtered} if there exists a $\mathbb{Z}^r$-filtration on it; that is, a family $\{F_\alpha(M)\}_{\alpha\in\mathbb{Z}^r}$ of $k$-subspaces of $M$ such that
\begin{enumerate}
 \item$F_\alpha(M)\subseteq F_{\alpha'}(M)$ if $\alpha<\alpha'$;
 \item$F_\alpha(B)F_{\alpha'}(M)\subseteq F_{\alpha+\alpha'}(M)$, for any $\alpha,\alpha'\in \mathbb{Z}^r$;
 \item$M=\bigcup_{\alpha\in\mathbb{Z}^r} F_\alpha(M)$.
\end{enumerate}
\end{definition}

Also there is an \emph{associated $\mathbb{Z}^r$-graded module} of $M$
$$
G^r(M)=\bigoplus_{\alpha\in\mathbb{Z}^r}\frac{F_\alpha(M)}{F_{<\alpha}(M)},
$$
where $F_{<\alpha}(M)=\bigcup_{\alpha'<\alpha}F_{\alpha'}(M)$. Notice that $G^r(M)\in\mathrm{GrMod\,}G^r(B)$.

Let $M$ be a $\mathbb{Z}^r$-filtered $B$-module. If $L$ is a submodule of $M$, there is an \emph{induced $\mathbb{Z}^r$-filtration} $\{F_\alpha (L)\}_{\alpha\in\mathbb{Z}^r}$ on $L$, where $F_\alpha(L)=L\bigcap F_\alpha(M)$. And an \emph{induced $\mathbb{Z}^r$-filtration} $\{F_\alpha (M/L)\}_{\alpha\in\mathbb{Z}^r}$ on $M/L$ is defined by $F_\alpha(M/L)=(F_\alpha(M)+L)/L$. We assume that the $\mathbb{Z}^r$-filtration on submodules and quotient modules is always the induced one in this paper.

\begin{definition}
For two $\mathbb{Z}^r$-filtered $M, N\in\mathrm{GrMod\,B}$, let $\phi$ be a $B$-homomorphism from $M$ to $N$. We say $\phi$ is \emph{$\mathbb{Z}^r$-filtered} if $\phi(F_\alpha(M))\subseteq F_\alpha(N)$ for any $\alpha\in\mathbb{Z}^r$.

Moreover, $\phi$ induces a $G^r(B)$-homomorphism from $G^r(M)$ to $G^r(N)$ denoted by $G^r(\phi)$.
\end{definition}

Besides $\phi$ is called \emph{strict} if $\phi(F_\alpha(M))=\phi(M)\bigcap F_\alpha(N)$ for all $\alpha\in\mathbb{Z}^r$. The strictness also yields $\phi(F_{<\alpha}(M))=\phi(M)\cap F_{<\alpha}(N)$ for all $\alpha$. When $r=1$, it also goes back to the usual situation.

\section{$\mathbb{Z}^r$-filtered algebras and associated $\mathbb{Z}^r$-graded algebras}\label{Z^n}

This section is devoted to set up a link between the connected graded algebra and its associated $\mathbb{Z}^r$-graded algebra. We define a $\mathbb{Z}^r$-filtration on a connected graded algebra related to a partition of the generator set, and discuss relevant homological properties. Using Gr\"obner basis theory, we prove two criterions for examining regularity and ring-properties from known algebras.

In \cite{ZZ1,ZZ2}, the authors proved the regularity and some other ring-properties of $\mathbb{N}$-filtered algebras can be examined from their associated graded algebras. However, $\mathbb{N}$-filtration is not enough to reduce
the complexity in general. Torrecillas and Lobillo studied GK-dimension and  global dimension of $\mathbb{N}^r$-filtered
algebras in \cite{TO} and \cite{TL}. To deal with general cases,
$\mathbb{Z}^r$-filtration should be much more selective. Here comes a question,
do the conclusions in \cite{ZZ1,ZZ2} still hold for some
$\mathbb{Z}^r$-filtration?

We consider 2-dimensional AS-regular algebras firstly. It only has
two types:
$$
\mathfrak{A}(q):=k\langle x,y\rangle/(xy-qyx),\quad\quad\quad \mathfrak{A}':=k\langle x,y\rangle/(yx-xy-x^2).
$$
The $\mathbb{Z}^2$-graded part $yx-xy$ in the relation of $\mathfrak{A}'$ is a special case of the relation of $\mathfrak{A}(q)$. In fact, $\mathfrak{A}(1)$ is an associated $\mathbb{Z}^2$-graded algebra of $\mathfrak{A}'$ for some $\mathbb{Z}^2$-filtration. Similar phenomenon exists in $S_2$ and $S_2'$ of 3-dimensional AS-regular algebras (see \cite{AS}). Those evidences inspire us to find a criterion in general.

\vspace{2mm}

\subsection{A $\mathbb{Z}^r$-filtration arising from a partition on the generator set}~

In order to guarantee that morphisms preserve the degrees, we need a special $\mathbb{Z}^r$-filtration. This $\mathbb{Z}^r$-filtration is natural with some homological properties.

Firstly, we give an admissible ordering on group $\mathbb{Z}^r$ as follows. Let $\alpha=(a_1, \cdots, a_r)$ and $\beta=(b_1, \cdots, b_r)$ be two arbitrary elements of $\mathbb{Z}^r$. We define $\alpha<\beta$ if one of the following two cases being satisfied
\begin{enumerate}
\item $\|\alpha\|<\|\beta\|$, or
\item $\|\alpha\|=\|\beta\|$ and there exists a $t$ ($1\le t\leq r$) such that $a_i=b_i$ for $i<t$ but $a_t<b_t$.
\end{enumerate}

Now let $A=k\langle X\rangle/I\/$ be a connected graded algebra, where $X$ is the minimal set of generators of $A$. Denote $X^*$ the free monoid generated by $X$ including $1$. There is a canonical projection
$$
\pi: k\langle X\rangle\rightarrow A.
$$

Given a positive integer $r$ ($1<r\le \#(X)$), we introduce a $\mathbb{Z}^r$-grading $\deg^r$ on $k\langle X\rangle$ as follows: let $\{X_1, X_2, \cdots, X_r\}$ be a partition of $X$, define $\deg^r x:=(\delta_{1i}, \delta_{2i}, \cdots, \delta_{ri})$ if $x\in X_i$ ($i=1, 2, \cdots, r$), where $\delta_{ij}$ is the Kronecker symbol.
Using this grading, we can get a $\mathbb{Z}^r$-filtration on $k\langle X\rangle$ defined by: $F_\alpha(k\langle X\rangle)=0$ if $\alpha<0$ and $F_\alpha(k\langle X\rangle)=\text{Span}_k\{ u\in X^*\;| \; \deg^ru \leq \alpha\}$ if $\alpha\ge 0$. That induces a $\mathbb{Z}^r$-filtration on $A$ defined by
$$
F_\alpha(A)=\pi(F_\alpha(k\langle X\rangle)),\quad \mathrm{for\; any} \; \alpha\in\mathbb{Z}^r.
$$
\textbf{Convention}: \textsl{From now on, we fix this $\mathbb{Z}^r$-filtration on $A$, and denote by $G^r(A)$ the associated $\mathbb{Z}^r$-graded algebra of $A$.}

Note that $A$ is generated in degree 1, the following is obvious.
\begin{lemma}\label{properties of Filtration on A}  The $\mathbb{Z}^r$-filtration on $A$ defined as above satisfies
\begin{enumerate}
\item$F_\alpha(A)\subseteq \bigoplus_{i\leq\|\alpha\|}A_i$ and $F_\alpha(A)\subseteq A_{\|\alpha\|}+F_{<\alpha}(A)$ for all $\alpha\in\mathbb{Z}^r$.
\item If $\alpha \notin\mathbb{N}^r$, then $F_\alpha(A)=F_{<\alpha}(A)$.
\item $G^r(A)$ is a connected  properly $\mathbb{Z}^r$-graded algebra.
\end{enumerate}
\end{lemma}

Let $P=\bigoplus_{i=1}^sAe_i$ be a finitely generated free $A$-module on the basis $\{e_i\}_{i=1}^s$. Take $\alpha_i\in\mathbb{Z}^r \;(i=1, 2,\cdots, s)$ such that $\|\alpha_i\|=\deg e_i$. We define a $\mathbb{Z}^r$-filtration on $P$ by
\begin{equation*}\label{filtration on free modules}
F_\alpha(P)=\bigoplus_{i=1}^s\Big( \sum_{\gamma+\alpha_i\leq\alpha}F_\gamma(A) \Big)e_i,\quad \alpha\in \mathbb{Z}^r.\tag{F1}
\end{equation*}
It is easy to check that $G^r(P)$, the associated $\mathbb{Z}^r$-graded module of $P$, is finitely generated and free as $G^r(A)$-module. We call $(P\,,\{\alpha_i\}_{i=1}^s)$ a \emph{$\mathbb{Z}^r$-filtered pair of free module}.

Conversely, let $\underline{P}=\bigoplus_{i=1}^sG^r(A)\underline{e}_i$ be a finitely generated free $G^r(A)$-module on the basis $\{\underline{e}_i\}_{i=1}^s$ with $\deg \underline{e}_i=\alpha_i$ for $i=1,2,\cdots, s$. Then there exists a $\mathbb{Z}^r$-filtered pair of free module $(P\,,\{\alpha_i\}_{i=1}^s)$ such that $G^r(P)=\underline{P}$, where $P=\bigoplus_{i=1}^sAe_i$ is a finitely generated free $A$-module. Moreover, we set $\deg e_i=\mathrm{tdeg\,}\underline{e}_i$ for $i=1, 2,\cdots, s$.

For other modules $M\in\mathrm{grmod}\,A$, some extra hypotheses of $\mathbb{Z}^r$-filtration on $M$ is required. For convenience, we introduce a $\mathbb{Z}^r$-filtration on $M$ by
\begin{equation*}\label{filtration on fg modules}
F_\alpha(M)=\sum_{i=1}^s\Big(\sum_{\gamma+\beta_i\leq\alpha}F_{\gamma}(A)\Big)\xi_i,\quad\mathrm{for\; all\; }\alpha\in\mathbb{Z}^r,\tag{F2}
\end{equation*}
where $M=\sum_{i=1}^sA\xi_i$ and $\beta_i=(0, \cdots, 0, \deg \xi_i)$ for any $i=1,2,\cdots,s$. This $\mathbb{Z}^r$-filtration on finitely generated modules is so called ``good'' $\mathbb{Z}^r$-filtration. It also assures that $G^r(M)\neq0$ if $M\neq0$. According to Lemma \ref{properties of Filtration on A}, this $\mathbb{Z}^r$-filtration on finitely generated module $M$ implies the following lemma:

\begin{lemma}\label{properties of Filtration modules}
Let $M$ be a finitely generated $A$-module.
\begin{enumerate}
\item $G^r(M)$ is a finitely generated  $G^r(A)$-module.
\item$F_\alpha(M)\subseteq \bigoplus_{i\leq\|\alpha\|}M_i$ and $F_\alpha(M)\subseteq M_{\|\alpha\|}+F_{<\alpha}(M)$ for all $\alpha\in\mathbb{Z}^r$.
\item Let $\alpha=(a_1,\cdots, a_{r-1}, a_r)\in\mathbb{Z}^r$ such that $(a_1, \cdots, a_{r-1})\notin\mathbb{N}^{r-1}$, then $F_{\alpha}(M)=F_{<\alpha}(M)$.
\item For every $i\in\mathbb{Z}$, there are only finite $\alpha\in\mathbb{Z}^r$ such that $\|\alpha\|=i$ and $F_\alpha(M)\neq F_{<\alpha}(M)$.
\end{enumerate}
\end{lemma}
Lemma \ref{properties of Filtration modules} implies that the $\mathbb{Z}^r$-filtration $\{F_\alpha(M)\}_{\alpha\in\mathbb{Z}^r}$ on finitely generated module $M$ is well-ordering with respect to the order by inclusion. Furthermore, for any $m\in M$, there exists $\alpha\in\mathbb{Z}^r$ such that $F_\alpha(M)\neq F_{<\alpha}(M)$ and $m\in F_\alpha(M)$. The $\mathbb{Z}^r$-filtration on modules in $\mathrm{grmod\,}A^o$ can be defined in a similar way.

In the sequel, all $\mathbb{Z}^r$-filtration on free modules and finitely generated modules is considered to be defined as (\ref{filtration on free modules}) and (\ref{filtration on fg modules}) respectively. With these
preparations, we turn to consider the homological aspect of them.
\begin{lemma}\label{exactnees original and graded}
Suppose $\mathbb{Z}^r$-filtered modules $M_1, M_2, M_3\in\mathrm{GrMod}\,A$. Consider $\mathbb{Z}^r$-filtered sequence
\begin{equation*}
M_1\xrightarrow{\varphi_1} M_2\xrightarrow{\varphi_2} M_3,\tag{$\natural$}
\end{equation*}
with $\varphi_2\varphi_1=0$ and the associated $\mathbb{Z}^r$-graded sequence:
\begin{equation*}
G^r(M_1)\xrightarrow{G^r(\varphi_1)} G^r(M_2)\xrightarrow{G^r(\varphi_2)} G^r(M_3).\tag{$G^r(\natural)$}
\end{equation*}

\begin{enumerate}
\item If sequence $(\natural)$ is exact and $\varphi_1,\varphi_2$ are strict, then sequence $(G^r(\natural))$ is exact.
\item Suppose $M_i\in\mathrm{grmod}\,A$ (i=1, 2, 3), then sequence $(G^r(\natural))$ is exact if and only if $(\natural)$ is exact and $\varphi_1, \varphi_2$ are strict .
\end{enumerate}
\end{lemma}
\begin{proof}
(a) Clearly $G^r(\varphi_2)G^r(\varphi_1)=0$. Let $m_2\in F_\alpha(M_2)\backslash F_{<\alpha}(M_2)$, and $0\ne \overline{m}_2\in G^r(M_2)$. Suppose $G^r(\varphi_2)(\overline{m}_2)=0$. This implies $\varphi_2(m_2)\in F_{<\alpha}(M_3)$. However, $\varphi_2(m_2)\in\varphi_2(F_{<\alpha}(M_2))$ by the strictness of $\varphi_2$. There exists $m_2'\in F_{\alpha'}(M_2)$ such that $m_2-m_2'\in\Ker\varphi_2$ where $\alpha'<\alpha$. Exactness of ($\natural$) and strictness of $\varphi_1$ yield
$$
\varphi_1(F_\alpha(M_1))=\Image(\varphi_1)\cap F_\alpha(M_2)=\Ker(\varphi_2)\cap F_\alpha(M_2).
$$
Thus $\varphi_1(m_1)=m_2-m_2'$ for some $m_1\in F_\alpha(M_1)$. Then $G^r(\varphi_1)(\overline{m}_1)=\overline{\varphi_1(m_1)}=\overline{m}_2$.
This shows $\Ker G^r(\varphi_2)\subseteq\Image G^r(\varphi_1)$.

(b) The sufficiency is a special case of (a). To get the necessity, we proceed it in two steps. The first step is to show the strictness. We need only to prove the strictness of $\varphi_2$ since a similar argument is valid for $\varphi_1$.

Choose $m_3\in F_\alpha(M_3)\cap\Image(\varphi_2)$ and $m_3\notin F_{<\alpha}(M_3)$. There exists $m_2\in M_2$ with degree $\|\alpha\|$ which belongs to $F_{\alpha'}(M_2)$ such that $\varphi_2(m_2)=m_3$. If $\alpha'>\alpha$, $G^r(\varphi_2)(\overline{m}_2)=\overline{\varphi_2(m_2)}=0$. By the exactness, there exists $m_1\in F_{\alpha'}(M_1)$ with degree $\|\alpha\|$ such that $G^r(\varphi_1)(\overline{m}_1)=\overline{\varphi_1(m_1)}=\overline{m}_2$. Thus $m_2'=m_2-\varphi_1(m_1)\in F_{<\alpha'}(M_2)$ such that
$\varphi_2(m_2')=m_3$. The proof is completed if $m_2'\in F_\alpha(M_2)$. Otherwise, there is $\alpha''>\alpha$ such that $m_2'\in F_{\alpha''}(M_2)$  and $F_{\alpha''}\neq F_{<\alpha''}$. Repeat this procedure, by Lemma \ref{properties of Filtration modules}(d), it stops in finite steps. Finally, we get $\widetilde{m}_2\in F_\alpha{(M_2)}$ such that $\varphi_2(\widetilde{m}_2)=m_3$.

The second step is exactness. Let $m_2\in F_\alpha(M_2)\backslash F_{<\alpha}(M_2)$ such that $\varphi_2(m_2)=0$. Then $G^r(\varphi_2)(\overline{m}_2)=0$. Hence $\overline{m}_2\in\Ker G^r(\varphi_2)$. There exists $m_1^{(1)}\in F_\alpha(M_1)\backslash F_{<\alpha}(M_1)$ satisfying $G^r(\varphi_1)(\overline{m_1^{(1)}})=\overline{\varphi_1(m_1^{(1)})}
=\overline{m}_2.$
Thus $m_2'=m_2-\varphi_1(m_1^{(1)})\in F_{<\alpha}(M_2)\cap \Ker \varphi_2$. There is $\alpha'<\alpha$ such that $m_2'\in F_{\alpha'}(M_2)$ and $F_{\alpha'}(M_2)\neq F_{<\alpha'}(M_2)$. Repeat this procedure, by Lemma \ref{properties of Filtration modules}(d) and $M_2$ is bounded below, there exist finite number of $m_1^{(1)},m_1^{(2)},\cdots,m_1^{(t)}$ such that $m_2=\varphi_1(\sum_{i=1}^tm_1^{(i)})$. Therefore $\Ker \varphi_2\subseteq\Image\varphi_1$.
\end{proof}

The following corollary tells that the properties of submodules can also be obtained from its associated $\mathbb{Z}^r$-graded version.
\begin{corollary}\label{submodules gr eq then they eq}
Let $M_1, M_2\in\mathrm{GrMod}\,A$ be $\mathbb{Z}^r$-filtered modules.
\begin{enumerate}
\item Suppose $\phi:M_1\rightarrow M_2$ is a strict $\mathbb{Z}^r$-filtered homomorphism. Then $\Image G^r(\phi)\cong G^r(\Image \phi)$ and $\Ker G^r(\phi)\cong G^r(\Ker \phi)$.
\item If $L$ is a submodules of $M_1$, then $G^r(M_1/L)\cong G^r(M_1)/G^r(L)$.
\item Suppose $M_1\in\mathrm{grmod\,}A$. If $L_1, L_2$ are two submodules of $M_1$ and $L_1\subseteq L_2$ such that $G^r(L_1)=G^r(L_2)$, then $L_1=L_2$.
\end{enumerate}
\end{corollary}
\begin{proof}
(a) and (b) are immediate results of Lemma \ref{exactnees original and graded}(a).

(c) Notice that the $\mathbb{Z}^r$-filtration on $L_1,L_2$ and $L_2/L_1$, induced from the one on $M_1$, is also well-ordering. Then $L_2/L_1=0$ follows from $G^r(L_2/L_1)=0$, and the latter follows from (b).
\end{proof}

Next lemma is a special case of \cite[Proposition 2.3, Chapter 2]{LHS}. However, it holds in the graded case and we are concentrated on the construction of morphism.
\begin{lemma}\label{morphisim in ass graded to original}
Let $M$ be a finitely generated $A$-module, and $\underline{P}=\bigoplus_{i=1}^sG^r(A)\underline{e}_i$ a finitely generated free $G^r(A)$-module with basis $\{\underline{e}_i\}_{i=1}^s$. Assume $\underline{\phi}:\underline{P}\rightarrow G^r(M)$ is a surjective morphism in $\mathrm{GrMod}\,G^r(A)$. Then there exist a finitely generated free $A$-module $P$ and a strict $\mathbb{Z}^r$-filtered surjection $\phi:P\rightarrow M$ in $\mathrm{GrMod}\,A$ such that $G^r(\phi)=\underline{\phi}$.
\end{lemma}
\begin{proof}
As mentioned above, set $P=\bigoplus_{i=1}^sAe_i$ on a basis of $\{e_i\}_{i=1}^s$ with $\deg e_i=\mathrm{tdeg\,} \underline{e}_i$ for $i=1, 2, \cdots, s$. $(P,\{\alpha_i\}_{i=1}^s)$ is a $\mathbb{Z}^r$-filtered pair of free modules such that $G^r(P)=\underline{P}$ where $\alpha_i=\deg\underline{e}_i\ (i=1,2,
\cdots,s)$.

Assume $\underline\phi(\underline{e}_i)=\overline{m}_i$, where $\overline{m}_i$ is a homogenous element in $G^r(M)_{\alpha_i}$ represented by $m_i$. Here, $m_i\in F_{\alpha_i}(M)\backslash F_{<\alpha_i}(M)$ and $m_i\in M_{\|\alpha_i\|}$ for $i=1,2,\cdots,s$. The existence of $m_i$ is  guaranteed by Lemma \ref{properties of Filtration modules}.

We define the morphism $\phi:P\rightarrow M$ in $\mathrm{GrMod}\,A$ by $\phi(e_i)=m_i$. It is easy to know $G^r(\phi)=\underline{\phi}$.
Since $\underline{\phi}$ is surjective, $\phi$ is a strict $\mathbb{Z}^r$-filtered surjection by Lemma \ref{exactnees original and graded}(b).
\end{proof}
The following lemma exhibits a construction of a free resolution for a $\mathbb{Z}^r$-filtered graded module from a free resolution of its associative $\mathbb{Z}^r$-graded module.
\begin{lemma}\label{resolution filtered to orginial}
Let $\mathbb{Z}^r$-filtered $M$ be a finitely generated $A$-module. $G^r(M)$ has a finite free resolution (the length is finite and each term in it is finitely generated):
$$
\xymatrix @M=2px@C=2pc@R=1pc{%
    	 0\ar[r]&{\underline{P}_m}\ar[r]^(0.4){\underline{d}_{m}}
    &{\underline{P}_{m-1}}\ar[r]^{\underline{d}_{m-1}}&\cdots
    \ar[r]^{\underline{d}_{2}}&{\underline{P}_1}\ar[r]^{\underline{d}_{1}}
    &{\underline{P}_0}\ar[r]^(0.4){\underline{d}_{0}}&{G^r(M)}\ar[r]&0
    ,}%
$$
where $\underline{P}_j=\bigoplus_{i=1}^{s_{j}}G^r(A)\underline{e}^j_i$ for $0\leq j\leq m$.
Then there exists a finite free resolution of $M$ in $\mathrm{GrMod}\,A$:
$$
    \xymatrix @M=2px@C=2pc@R=1pc{%
    	 0\ar[r]&{P_m}\ar[r]^(0.4){{d}_{m}}&{P_{m-1}}\ar[r]^{{d}_{m-1}}
    &\cdots\ar[r]^{{d}_{2}}&{P_1}\ar[r]^{{d}_{1}}&{P_0}\ar[r]^(0.4){{d}_{0}}&{M}\ar[r]&0
    ,}%
$$
where $P_j=\bigoplus_{i=1}^{s_{j}}A{e}^{j}_i$ with $\deg {e}^j_i=\mathrm{tdeg\,}\underline{e}^j_i$ and $(P_j, \{\deg \underline{e}_i^j\}_{i=1}^{s_j})$ is a $\mathbb{Z}^r$-filtered pair of free module such that $G^r(P_j)\cong \underline{P}_j$ and $G^r({d}_j)=\underline{d}_j$ for all $1\leq i\leq s_j, 0\leq j\leq m.$

As a consequence, $gl\dim A\leq gl\dim G^r(A)$.
\end{lemma}
\begin{proof}
The proof is similar to \cite[Theorem 2.7]{TL} and \cite[Proposition 2.5, Chapter 2]{LHS}. But we need to notice that ${d}_i$ is constructed as in Lemma \ref{morphisim in ass graded to original} which preserves degrees for all $i=1, 2, \cdots, m$.
\end{proof}

In the sequel, we use the following notations if there is no confusion:
\begin{eqnarray*}
&&(-)^\vee:=\uHom_A(-,A);\\
&&(-)^\vee:=\uHom_{G^r(A)}(-,G^r(A)).
\end{eqnarray*}

As usual, $\mathbb{Z}^r$-filtration for functor $(-)^\vee$ is defined below
$$
F_\alpha(P^\vee)=\{f\in P^\vee\;|\;\;f(F_{\alpha'}(P))\subseteq F_{\alpha'+\alpha}(A)\mathrm{\; for\;all\; \alpha'\in\mathbb{Z}^r}\},
$$
where $P=\bigoplus_{i=1}^sAe_i$ is a finitely generated free $A$-module on basis $\{e_i\}_{i=1}^s$. However, there is an isomorphism $                               \theta:P^\vee\cong P'$ in $\mathrm{grmod}\,A^o$ , where $P'=\bigoplus_{i=1}^se'_iA$ is a free $A^o$-module on a basis $\{e'_i\}_{i=1}^s$ with $\deg e'_i=-\deg e_i$ for $i=1, 2,\cdots, s$. It is easy to check that $\theta$ is a strict $\mathbb{Z}^r$-filtered isomorphism. Thus, the $\mathbb{Z}^r$-filtration above also satisfies Lemma \ref{properties of Filtration modules}.

\begin{lemma}\label{dual for graded} Let
$(\underline{P}_i,\, \underline{d}_i)$ and $(P_i,\, d_i)$ be defined as in Lemma \ref{resolution filtered to orginial} for $i=1, 2, \cdots, m$. Then $G^r(P_i^\vee)\cong \underline{P}_i^\vee$ and  $G^r(d_i\!^\vee)=\underline{d}_i\!^\vee$ for $i=1, 2, \cdots, m$.
\end{lemma}
\begin{proof}
Since $P_i^\vee$ is a finitely generated free module, $\underline{P}_i^\vee\cong G^r(P_i^\vee)$ is a easy result by Lemma \ref{properties of Filtration modules}. And the other one can be verified straightforwardly by Lemma \ref{morphisim in ass graded to original} and the isomorphism $\theta$.
\end{proof}
\vspace{2mm}

\subsection{The regularity}~

Now, we give the regularity criterion for a connected graded algebra.

Let $A=k\langle X\rangle/I$ be a connected graded algebra. We keep the $\mathbb{Z}^r$-filtration on $A$ defined in last subsection. Actually, this $\mathbb{Z}^r$-filtration is equivalent to a $\mathbb{Z}^r$-grading on $k\langle X\rangle$ such that $G^r(A)$ is a $\mathbb{Z}^r$-graded algebra. Provided an appropriate $\mathbb{Z}^r$-grading, one may derive some properties of $A$ from $G^r(A)$.

\begin{theorem}\label{key theo 1}
Let $A=k\langle X\rangle/I$ be a connected graded algebra. If $G^r(A)$ is AS-regular for an appropriate $\mathbb{Z}^r$-grading on $k\langle X\rangle$, then $A$ is AS-regular.
\end{theorem}
\begin{proof}
By Lemma \ref{resolution filtered to orginial}, we know $gl\dim A\leq gl\dim G^r(A)$ is finite.

Notice that $G^r(A)$ can be seen as a $\mathbb{N}^r$-graded algebra which does not change the GK-dimension. Thus $\mathrm{GK}\dim A=\mathrm{GK} \dim G^r(A)$ is finite by \cite[Theorem 2.8]{TO}.

It remains to show that $A$ is Gorenstein. Since $G^r(A)$ is AS-regular, there exists a minimal free resolution of $_{G^r(A)}k$:
$$
\xymatrix @M=2px@C=2pc@R=1pc{%
    	 0\ar[r]&{\underline{P}_n}\ar[r]^(0.4){\underline{d}_{n}}
    &{\underline{P}_{n-1}}\ar[r]^{\underline{d}_{n-1}}&\cdots\ar[r]^{\underline{d}_{2}}
    &{\underline{P}_1}\ar[r]^{\underline{d}_{1}}&{\underline{P}_0}
    \ar[r]^(0.4){\underline{d}_{0}}&{_{G^r(A)}k}\ar[r]&0
    ,}%
$$
where $\underline{P}_j=\bigoplus_{i=1}^{s_{j}}G^r(A)\underline{e}^{j}_i$ for all $0\leq j\leq n.$

It is easy to know $G^r(k)\cong k$ as $G^r(A)$-modules. By Lemma \ref{resolution filtered to orginial}, there exists a finite free resolution of $_Ak$ in $\mathrm{GrMod}\,A$:
\begin{equation*}
    \xymatrix @M=2px@C=2pc@R=1pc{%
    	 0\ar[r]&{P_n}\ar[r]^(0.4){{d}_{n}}&{P_{n-1}}\ar[r]^{{d}_{n-1}}
    &\cdots\ar[r]^{{d}_{2}}&{P_1}\ar[r]^{{d}_{1}}&{P_0}\ar[r]^(0.4){{d}_{0}}&{_Ak}\ar[r]&0
    ,}%
\end{equation*}
where $P_j=\bigoplus_{i=1}^{s_{j}}A{e}^{j}_i$ with $\deg{e}^j_i=\mathrm{tdeg\,}\underline{e}^j_i$ and $(P_j,\{\deg \underline{e}_i^j\}_{i=1}^{s_j})$ is a $\mathbb{Z}^r$-filtered pair of free module such that $G^r(P_j)\cong \underline{P}_j$ and $G^r({d}_j)=\underline{d}_j$ for all $1\leq i\leq s_j, \, 0\leq j\leq n.$

The regularity of $G^r(A)$ implies $\Image(\underline{d}_{i-1}\!^\vee)=\Ker(\underline{d}_i\!^\vee)$. Note that $G^r({d}_i\!^\vee)=\underline{d}_i\!^\vee$ and $G^r(P_i^\vee)\cong \underline{P}_i^\vee$ by Lemma \ref{dual for graded} for $i=1, 2,\cdots, n$. Thus $\uExt_A^i(k,A)=0$ for all $i<n$ by Lemma \ref{exactnees original and graded}(b). Moreover, ${d}_n\!^\vee$ is strict. Now we turn to compute $\uExt_A^n(k,A)$.
$$
\uExt_{G^r(A)}^n(k,G^r(A))=\underline{P}_n^\vee/\Image(\underline{d}_n\!^\vee)\cong G^r(P_n^\vee)/\Image(G^r({d}_{n}\!^\vee))\cong k(l)
$$
for some $l\in \mathbb{Z}^r$. By Corollary \ref{submodules gr eq then they eq}, we have $$G^r(\uExt_A^n(k,A))=G^r(P_n^\vee/\Image({d}_{n}\!^\vee))\cong G^r(P_n^\vee)/\Image(G^r({d}_{n}\!^\vee))\cong k(l),$$
where the $\mathbb{Z}^r$-filtration on $\uExt_A^n(k,A)$ is induced by the one on $P_n^\vee$.

Hence, $F_{-l}(\uExt_A^n(k,A))/F_{<-l}(\uExt_A^n(k,A))\cong k$ and $F_\alpha(\uExt_A^n(k,A))=F_{<\alpha}(\uExt_A^n(k,A))$ for all $\alpha\in\mathbb{Z}^r$ except for $\alpha=-l$. Since the $\mathbb{Z}^r$-filtration on $P_n^\vee$ satisfies Lemma \ref{properties of Filtration modules}, we know $\uExt_A^n(k,A)\cong k(\|l\|)$. In addition, $gl\dim A=gl\dim G^r(A)$.
\end{proof}

To make the regularity criterion theorem above available in practice, a good way is to use Gr\"obner theory. We review noncommutative Gr\"obner basis theory briefly, a detailed treatment can be found in \cite{LHS}. We firstly choose an arbitrary monomial ordering $\prec$ on $X^*$. This induces a $\mathbb{Z}^r$-graded admissible ordering $\prec_{\mathbb{Z}^r}$ on $X^*$: for $u, v\in X^*$, $u\prec_{\mathbb{Z}^r}v$ is defined by
\begin{enumerate}
\item $\deg^r(u)<\deg^r(v)$, or
\item $\deg^r(u)=\deg^r(v)$ and $u\prec v$.
\end{enumerate}

For a nonzero polynomial $f\in k\langle X\rangle$, we can write $f=\sum_{i=1}^qf_i$, where each nonzero $f_i$ is $\mathbb{Z}^r$-homogenous with $\deg^rf_i=\alpha_i$ and $\alpha_1<\alpha_2<\cdots<\alpha_q$,  $f_q$ is called the \emph{leading homogenous polynomial} of $f$, denoted by $LH(f)$. Let $\mathcal{G}$ be the reduced monic Gr\"obner basis of $I$ under admissible ordering $\prec_{\mathbb{Z}^r}$, and $LH(\mathcal{G})=\{LH(f)\;|\;f\in \mathcal{G}\}$.

With the above preparations, now we are in position to prove Theorem \ref{first theo}.

\begin{proof}[Proof of Theorem \ref{first theo}]
Due to the observation above Theorem \ref{key theo 1}, there exists a partition on generator set $X$ corresponding to the $\mathbb{Z}^r$-grading on $k\langle X\rangle$. This partition induces a  $\mathbb{Z}^r$-filtration on $A$ as defined in Section 2.1, and $G^r(A)$ is the associated $\mathbb{Z}^r$-graded algebra. Here the ordering $<$ on $\mathbb{Z}^r$ is the top priority in $\prec_{\mathbb{Z}^r}$. And $\mathcal{G}$ is the reduced Gr\"obner basis of $I$ with respect to $\prec_{\mathbb{Z}^r}$.  From \cite[Theorem 2.3, Chapter 4]{LHS}, we know
$$
G^r(A)\cong k\langle X\rangle/(LH(\mathcal{G}))
$$
as $\mathbb{Z}^r$-graded algebras.
Thus $A$ is AS-regular by Theorem \ref{key theo 1}.
\end{proof}

\begin{remark}\label{deformation}
Theorem \ref{first theo} provides a possible generalized deformation from known $\mathbb{Z}^r$-graded AS-regular algebras; that is, by adding some appropriate low-terms to the relations, one may produce some new classes of AS-regular algebras.
\end{remark}

\begin{question}
Suppose that $A=k\langle X\rangle/ I$ is an AS-regular algebra, does there exist an appropriate $\mathbb{Z}^r$-grading on $k\langle X\rangle$ such that $k\langle X\rangle/(LH(\mathcal{G}))$ is AS-regular?
\end{question}

\vspace{2mm}

\subsection{Ring-theoretic and homological properties}~

AS-regular algebras obtained so far all have nice ring-theoretic and homological properties, such as noetherian, strongly noetherian, and Auslander regular. In this subsection, we show that those properties also hold if their associated $\mathbb{Z}^r$-graded algebras have.

\begin{theorem}\label{nsnac}
Let $A=k\langle X\rangle/I$ be a connected graded algebra.  If $G^r(A)$ is strongly noetherian and Auslander regular for an appropriate $\mathbb{Z}^r$-grading on $k\langle X\rangle$, then so is $A$.
\end{theorem}

Before proving this theorem, we need some lemmas.

Firstly, we set the definition of $\mathbb{Z}^r$-filtration on tensor product. Let $A_1$ and $A_2$ be two algebras with $A_1$ being a $\mathbb{Z}^r$-filtered algebra. We introduce a $\mathbb{Z}^r$-filtration on $A_1\otimes A_2$ by
$$
F_\alpha(A_1\otimes A_2)=F_\alpha(A_1)\otimes A_2\quad \text{ for all }\alpha\in\mathbb{Z}^r.
$$
\begin{remark}\label{remark fro filtration on not connected}
Suppose $A_1$ is a connected graded algebra and $A_2$ is regard as a graded algebra concentrated in degree 0. The $\mathbb{Z}^r$-filtration on $A_1$ is the one defined in Section 2.1, then the $\mathbb{Z}^r$-filtration on $A_1\otimes A_2$ satisfies Lemma \ref{properties of Filtration on A}(a,b). And $\mathbb{Z}^r$-filtration on modules in $\mathrm{grmod\,}A_1\otimes A_2$ is defined as (\ref{filtration on fg modules}). Then Lemma \ref{properties of Filtration modules}, Lemma \ref{exactnees original and graded} and Corollary \ref{submodules gr eq then they eq} still hold in the category $\mathrm{grmod\,}(A_1\otimes A_2)$.
\end{remark}
\begin{lemma}\label{tensor}
Let $A_1$ and $A_2$ be two algebras where $A_1$ is a $\mathbb{Z}^r$-filtered algebras, then $G^r(A_1\otimes A_2)\cong G^r(A_1)\otimes A_2$.
\end{lemma}
\begin{proof}
For any $\alpha\in\mathbb{Z}^r$, there exists an exact sequence as vector space,
$$
0\rightarrow F_{<\alpha}(A_1)\rightarrow F_\alpha(A_1)\rightarrow G^r(A_1)_\alpha\rightarrow 0.
$$
Note that $A_2$ is flat as $k$-module. Hence, acting $-\otimes A_2$ on that sequence,
$$
0\rightarrow F_{<\alpha}(A_1)\otimes A_2\rightarrow F_\alpha(A_1)\otimes A_2\rightarrow G^r(A_1)_\alpha\otimes A_2\rightarrow 0.
$$
is still exact, which implies
$$
G^r(A_1)_\alpha\otimes A_2\cong\frac{F_\alpha(A_1)\otimes A_2}{F_{<\alpha}(A_1)\otimes A_2}=G^r(A_1\otimes A_2)_\alpha.
$$

It is easy to check that $G^r(A_1\otimes A_2)\cong G^r(A_1)\otimes A_2$ as $\mathbb{Z}^r$-graded algebras.
\end{proof}

\begin{lemma}\label{kernel}
Let $A$ be a connected graded algebra. If $\mathbb{Z}^r$-filtered $A$-module $M$ has a finite free resolution, then $G^r(\uExt_A^i(M,A))$ is a subquotient of $\uExt_{G^r(A)}^i(G^r(M),G^r(A))$ for any $i\geq 0$.
\end{lemma}
\begin{proof}
We claim that for any $\mathbb{Z}^r$-filtered homomorphism $\phi: N_1\rightarrow N_2$, $G^r(\Ker\phi)\subseteq\Ker(G^r(\phi))$ and $\Image(G^r(\phi))\subseteq G^r(\Image\phi)$, where $N_1,N_2$ are two $\mathbb{Z}^r$-filtered $A$-modules. If it is the case, the conclusion follows from Lemma \ref{resolution filtered to orginial}.

Now we verify the claim. For any $\alpha\in\mathbb{Z}^r$,
$$
G^r(\Ker\phi)_\alpha=\frac{\Ker\phi\bigcap F_\alpha(N_1)}{\Ker\phi\bigcap F_{<\alpha}(N_1)}\cong\frac{\Ker\phi\bigcap F_\alpha(N_1)+F_{<\alpha}(N_1)}{F_{<\alpha}(N_1)}
$$
However,
$$
\Ker(G^r(\phi))_\alpha=\frac{\Ker\phi\bigcap F_\alpha(N_1)+\phi^{-1}(F_{<\alpha}(N_2))\bigcap F_\alpha(N_1)+F_{<\alpha}(N_1)}{F_{<\alpha}(N_1)}.
$$
where $\phi^{-1}(F_{<\alpha}(N_2))=\{n_1\in N_1\;|\; \phi(n_1)\in F_{<\alpha}(N_2)\}$. Obviously, $G^r(\Ker\phi)\subseteq\Ker(G^r(\phi))$.

The proof for $\Image\phi$ is similar.
\end{proof}
We now recall the definition of $j$-number of modules. Let $A$ be a $\mathbb{Z}^r$-graded algebra and $M\in\mathrm{GrMod}\,A$,
$$
j_A(M)=\inf\{i\;|\;\uExt_A^i(M,A)\neq0\}.
$$

\begin{lemma}\label{jnumber}
Let $A$ be a connected graded algebra. If $G^r(M)$ has a finite free resolution for $\mathbb{Z}^r$-filtered $M\in\mathrm{grmod}\,A$ $($resp. $\mathrm{grmod}\,A^o$$)$, then $j_A(M)\geq j_{G^r(A)}(G^r(M))$ $($resp. $j_{A^o}(M)\geq j_{G^r(A)^o}(G^r(M))$$)$.
\end{lemma}
\begin{proof}
Assume $G^r(M)$ has a finite free resolution
$$
\xymatrix @M=2px@C=2pc@R=1pc{%
    	    0\ar[r]&{\underline{P}_m}\ar[r]^(0.4){\underline{d}_{m}}
    &{\underline{P}_{m-1}}\ar[r]^{\underline{d}_{m-1}}
    &\cdots\ar[r]^{\underline{d}_{2}}&{\underline{P}_1}
    \ar[r]^{\underline{d}_{1}}&{\underline{P}_0}\ar[r]^(0.4)
    {\underline{d}_{0}}&{G^r(M)}\ar[r]&0
    ,}%
$$
where $\underline{P}_j=\bigoplus_{i=1}^{s_{j}}G^r(A)\underline{e}^{j}_i$ for $0\leq j\leq m$. By Lemma \ref{resolution filtered to orginial}, $M$ has a free resolution
$$
    \xymatrix @M=2px@C=2pc@R=1pc{%
    	 0\ar[r]&{P_m}\ar[r]^(0.4){{d}_{m}}&{P_{m-1}}\ar[r]^{{d}_{m-1}}
    &\cdots\ar[r]^{{d}_{2}}&{P_1}\ar[r]^{{d}_{1}}&{P_0}\ar[r]^(0.4){{d}_{0}}
    &{M}\ar[r]&0
    ,}%
$$
where $P_j=\bigoplus_{i=1}^{s_{j}}A{e}^{j}_i$ with $\deg{e}^j_i=\mathrm{tdeg\,}\underline{e}^j_i$ and $(P_j,\{\deg \underline{e}_i^j\}_{i=1}^{s_j})$ is a $\mathbb{Z}^r$-filtered pair of free module such that $G^r(P_j)\cong \underline{P}_j$ and $G^r({d}_j)=\underline{d}_j$ for $1\leq i\leq s_j, 0\leq j\leq m.$

Put $t=j_{G^r(A)}(G^r(M))$; that is, the following sequence is exact
$$
    \xymatrix @M=2px@C=2pc@R=1pc{%
    	 {\underline{P}_t^\vee}&{\underline{P}_{t-1}^\vee}
    \ar[l]_{\underline{d}_t\!^\vee}&{\cdots\cdots}
    \ar[l]_{\underline{d}_{t-1}\!^\vee}&{\underline{P}_0^\vee}\ar[l]_(0.4)
    {\underline{d}_0\!^\vee}&0\ar[l]
    }%
$$

By Lemma \ref{dual for graded} and Lemma \ref{exactnees original and graded}(b), we can know the sequence
$$
    \xymatrix @M=2px@C=2pc@R=1pc{%
    	    {P_t^\vee}&{P_{t-1}^\vee}\ar[l]_{{d}_t\!^\vee}&{\cdots\cdots}
    \ar[l]_{{d}_{t-1}\!^\vee}&{P_0^\vee}\ar[l]_(0.4){{d}_0\!^\vee}&0\ar[l]
    }%
$$
is also exact, which implies $j_A(M)\geq t$.
\end{proof}

\begin{proof}[Proof of Theorem \ref{nsnac}] Assume $G^r(A)$ is strongly noetherian. For every commutative noetherian algebra $B$, $G^r(A\otimes B)$ is noetherian by Lemma \ref{tensor}. $A\otimes B $ is noetherian by Remark \ref{remark fro filtration on not connected} and Corollary \ref{submodules gr eq then they eq} immediately.

If $G^r(A)$ is Auslander regular, so $A$ is also noetherian with finite global dimension. For any $M\in\mathrm{grmod}\,A$ and $i\in\mathbb{N}$, $N$ is an $A^o$-submodule of $\uExt^i_A(M,A)$. We need to show $j_{A^o}(N)\geq i$.

    $G^r(N)$ is a subquotient of $\uExt_{G^r(A)}^i(G^r(M),G^r(A))$ by Lemma \ref{kernel}. Therefore, the Auslander condition implies $j_{G^r(A)^o}(G^r(N))\geq i$. However, $j_{A^o}(N)\geq j_{G^r(A)^o}(G^r(N))\geq i$ by Lemma \ref{jnumber}.

    The right ones can be verified similarly.
\end{proof}

By means of Gr\"obner basis,  there is a corollary as the criterion for the regularity.

\begin{corollary}\label{nsnac1}
Let $A=k\langle X\rangle/I$ be a connected graded algebra. Suppose $\mathcal{G}$ is the reduced Gr\"obner basis of $I$ with respect to an admissible ordering $\prec_{\mathbb{Z}^r}$ for some $\mathbb{Z}^r$-grading on $k\langle X\rangle$. If the $\mathbb{Z}^r$-graded algebra ${k\langle X\rangle/(LH(\mathcal{G}))}$ is strongly noetherian and Auslander regular, then so is $A$.
\end{corollary}

\begin{remark}
We fail to prove that $A$ is Cohen-Macaulay if $G^r(A)$ is Cohen-Macaulay. It is equivalent to prove $j_{G^r(A)}(G^r(M))=j_A(M)$ for any $M\in\mathrm{grmod\,}A$. We conjecture it is true. For the class of AS-regular algebras $\mathcal{J}$, we will prove it directly in Section 5.
\end{remark}

\section{$A_\infty$-algebra structure of Jordan type}
From this section, we turn to the AS-regular algebras of type $(12221)$.  As mentioned in the introduction, we hope to classify the AS-regular algebras whose Frobenius data is in the Jordan case. We first review the $A_\infty$-algebra structures on the Ext-algebra of the type $(12221)$, the readers may find the details in \cite{LPWZ2}. After that, we concentrate on analyzing and solving the equations getting from the Stasheff identities in Jordan case.

In our case, $A$ is generated by 2 elements $x_1$ and $x_2$ with two relations $r_3$ and $r_4$ whose degrees are 3 and 4, respectively. Denote $E:=E(A)$ the $A_\infty$-Ext-algebras of $A$.
\vspace{2mm}

\subsection{$A_\infty$-Ext-algebras of type (12221)}~

Notice that our $A_\infty$-algebra structures satisfy the strictly unital condition, all multiplications and Stasheff identities can be described without $E^0=k$.
\vspace{2mm}
\subsubsection{Multiplications}~
\vspace{2mm}

According to the minimal resolution (\ref{resolution}) of trivial module $_Ak$, we know
\begin{equation*}
E\cong k\oplus E^1_{-1}\oplus E^2_{-3}\oplus E^2_{-4}\oplus  E^3_{-6}\oplus E^4_{-7},
\end{equation*}
where $\dim E^1_{-1}=\dim E^3_{-6}=2,~\dim E^2_{-3}=\dim E^2_{-4}=\dim E^4_{-7}=1.$

As stated above, all $m_n$ preserves Adams degree. After straightforward computation, we have $m_n=0$ except for $n=2, 3, 4$. The non-trivial multiplications $m_2, m_3, m_4$ are described explicitly in \cite{LPWZ2}. For the sake of computation, we copy them below.

$\bullet \; m_2:$
The possible non-trivial actions of $m_2$ on $E^{\otimes 2}$ are
\begin{eqnarray*}
&E^1_{-1}\otimes E^3_{-6}\rightarrow E^4_{-7},~~~&E^3_{-6}\otimes E^1_{-1}\rightarrow E^4_{-7},\\
&E^2_{-3}\otimes E^2_{-4}\rightarrow E^4_{-7},~~~&E^2_{-4}\otimes E^2_{-3}\rightarrow E^4_{-7}.
\end{eqnarray*}

By Lemma \ref{ASFrobenius}£¬$E$ is a Frobenius algebra. The Frobenius structure on $E$ can be described as follows.
There exists a basis $\{\beta_1,\beta_2\}$ of $E^1_{-1}$, a basis $\{\gamma_1\}$ of $E^2_{-3}$, a basis $\{\gamma_{2}\}$ of $E^2_{-4}$, a basis $\{\xi_1,\xi_2\}$ of $E^3_{-6}$, and a basis $\{\eta\}$ of $E^4_{-7}$ such that
\begin{equation*}
\begin{aligned}
&\gamma_1\gamma_2=\eta,&&\gamma_2\gamma_1=t\eta,&&t\in k,\\
&\beta_i\xi_j=\delta_{ij}\eta,&&\xi_i\beta_j=r_{ij}\eta,&&r_{ij}\in k.
\end{aligned}
\end{equation*}
where $t\neq 0, \mathcal{R}=(r_{ij})$ is nonsingular. $(\mathcal{R},t)$ is called the \emph{Frobenius data} of $E$.

Since $k$ is algebraic closed, $\mathcal{R}$ is similar to a diagonal matrix or a Jordan block; that is,
\begin{equation*}
\left(
\begin{array}{cc}
g_1&0\\
0  &g_2
\end{array}
\right)
\quad\quad \text{or} \quad\quad
\left(
\begin{array}{cc}
g&1\\
0  &g
\end{array}
\right).
\end{equation*}
We will focus on the later case which is called \emph{Jordan type}.

$\bullet \; m_3:$ Possible nonzero components of $m_3$ on $E^{\otimes 3}$ are
\begin{eqnarray*}
~&(E^1_{-1})^{\otimes 3}\rightarrow E^2_{-3},&~\\
(E^1_{-1})^{\otimes 2}\otimes E^2_{-4}\rightarrow E^3_{-6},&E^1_{-1} \otimes E^2_{-4}\otimes E^1_{-1} \rightarrow E^3_{-6}, &E^2_{-4}\otimes (E^1_{-1})^{\otimes 2}\rightarrow E^3_{-6},\\
E^1_{-1}\otimes(E^2_{-3})^{\otimes2}\rightarrow E^4_{-7},&E^2_{-3}\otimes E^1_{-1}\otimes E^2_{-3}\rightarrow E^4_{-7},&(E^2_{-3})^{\otimes2}\otimes E^1_{-1}\rightarrow E^4_{-7}.
\end{eqnarray*}
For $1\leq i,j,k\leq2$, we have
\begin{equation*}
\begin{array}{cr}
\multicolumn{2}{c}{m_3(\beta_i,\beta_j,\beta_k)=a_{ijk}\gamma_1,}\\
m_3(\beta_i,\beta_j,\gamma_2)=b_{13ij}\xi_1+b_{23ij}\xi_{2},&m_3(\beta_i,\gamma_1,\gamma_1)=c_{1i}\eta,\\
m_3(\beta_i,\gamma_2,\beta_j)=b_{12ij} \xi_1+b_{22ij}\xi_{2},&m_3(\gamma_1,\beta_i,\gamma_1)=c_{2i}\eta,\\
m_3(\gamma_2,\beta_i,\beta_j)=b_{11ij}\xi_1+b_{21ij}\xi_{2},&m_3(\gamma_1,\gamma_1,\beta_i)=c_{3i}\eta.
\end{array}
\end{equation*}
where the coefficients are scalars in $k$.

$\bullet \; m_4:$
The possible non-trivial actions of $m_4$ on $E^{\otimes 4}$ are
\begin{equation*}
\begin{array}{cccccc}
~&~~~~~~~\qquad\qquad\quad\quad\quad&\multicolumn{2}{r}{(E^1_{-1})^{\otimes 4}\rightarrow E^2_{-4},}&~&~\\
\multicolumn{3}{l}{(E^1_{-1})^{\otimes3}\otimes E^2_{-3}\rightarrow E^3_{-6},}
&\multicolumn{3}{r}{E^2_{-3}\otimes(E^1_{-1})^{\otimes3}\rightarrow E^3_{-6},}\\
\multicolumn{3}{c}{E^1_{-1}\otimes E^2_{-3}\otimes (E^1_{-1})^{\otimes2}\rightarrow E^3_{-6},}
&\multicolumn{3}{c}{(E^1_{-1})^{\otimes2}\otimes E^2_{-3}\otimes E^1_{-1}\rightarrow E^3_{-6}.}
\end{array}
\end{equation*}
For $1\leq i,j,k,h\leq2$, we have
\begin{equation*}
\begin{array}{cc}
\multicolumn{2}{c} {m_4(\beta_i,\beta_j,\beta_k,\beta_h)=v_{ijkh}\gamma_2,}\\
m_4(\beta_i,\beta_j,\beta_k,\gamma_1)=u_{14ijk}\xi_1+u_{24ijk}\xi_{2},&m_4(\beta_i,\beta_j,\gamma_1,\beta_k)=u_{13ijk}\xi_1+u_{23ijk}\xi_{2},\\
m_4(\beta_i,\gamma_1,\beta_j,\beta_k)=u_{12ijk}\xi_1+u_{22ijk}\xi_{2},&m_4(\gamma_1,\beta_i,\beta_j,\beta_k)=u_{11ijk}\xi_1+u_{21ijk}\xi_{2}.
\end{array}
\end{equation*}
where the coefficients are scalars in $k$.
\vspace{2mm}
\subsubsection{Stasheff identities for the $A_\infty$-algebra $E$}~

\vspace{2mm}
The non-trivial Stasheff identities are just \textsf{SI(4)}, \textsf{SI(5)}, \textsf{SI(6)}.

$\bullet$ \textsf{SI(4)}: Since $m_1=0$, \textsf{SI(4)} becomes
$$
m_3(m_2\otimes id^{\otimes 2}-id\otimes m_2\otimes id+ id^{\otimes 2}\otimes m_2)-m_2(m_3\otimes id+id\otimes m_3)=0.
$$
Applying it to the basis of $E$, the non-trivial ones give the relationships between the coefficients.
\begin{equation*}\label{SI(4a)}
\begin{array}{lll}
a_{ijk}=b_{i3jk},&b_{i2jk}=\sum_{s=1}^2r_{sk}b_{s3ij},&~\\
b_{i1jk}=\sum_{s=1}^2r_{sk}b_{s2ij},&-ta_{ijk}=\sum_{s=1}^2r_{sk}b_{s1ij},&\raisebox{1.5ex}[0pt]{ for $1\leq i,j,k\leq2.$}
\end{array}\tag{SI(4a)}
\end{equation*}
Immediately, we have
\begin{equation*}\label{SI(4b)}
-ta_{ijk}=\sum_{s,t,u=1}^2r_{sk}r_{tj}r_{ui}a_{uts}, \tag{SI(4b)}\quad\mathrm{for\;}1\leq i,j,k\leq2.
\end{equation*}

$\bullet$ \textsf{SI(5)}: The Stasheff identity \textsf{SI(5)} is equivalent to
\begin{eqnarray*}
&&m_4(m_2\otimes id^{\otimes3}-id\otimes m_2\otimes id^{\otimes2}+id^{\otimes2}\otimes m_2\otimes id-id^{\otimes3}\otimes m_2)\\
&&\quad +m_3(m_3\otimes id^{\otimes2}+id\otimes m_3\otimes id+id^{\otimes2}\otimes m_3)+m_2(m_4\otimes id-id\otimes m_4)=0.
\end{eqnarray*}
Then it follows that \textsf{SI(5)} holds if and only if for $1\leq i,j,k,h\leq2$,
\begin{equation*}\label{SI(5a)}
\begin{array}{l}
a_{ijk}c_{2h}-a_{jkh}c_{1i}+tv_{ijkh}-u_{i4jkh}=0,\\
a_{ijk}c_{3h}+r_{1h}u_{14ijk}+r_{2h}u_{24ijk}-u_{i3jkh}=0,\\
r_{1h}u_{13ijk}+r_{2h}u_{23ijk}-u_{i2jkh}=0,\\
c_{1i}a_{jkh}-r_{1h}u_{12ijk}-r_{2h}u_{22ijk}+u_{i1jkh}=0,\\
a_{jkh}c_{2i}-a_{ijk}c_{3h}-r_{1h}u_{11ijk}-r_{2h}u_{21ijk}+v_{ijkh}=0.
\end{array}\tag{SI(5a)}
\end{equation*}

$\bullet$ \textsf{SI(6)}: The Stasheff identity \textsf{SI(6)} becomes
\begin{eqnarray*}
&&m_4(-m_3\otimes id^{\otimes3}-id\otimes m_3\otimes id^{\otimes2}-id^{\otimes3}\otimes m_3\otimes id-id^{\otimes3}\otimes m_3)\\
&&\quad +m_3(m_4\otimes id^{\otimes2}-id\otimes m_4\otimes id+id^{\otimes2}\otimes m_4)=0.
\end{eqnarray*}
Applying it to the basis of $E$, all are trivial except for $(\beta_i,\beta_j,\beta_k,\beta_h,\beta_m,\beta_n)$. We obtain
\begin{equation*}\label{SI(6a)}
\begin{array}{cc}
-a_{ijk}u_{s1hmn}+a_{jkh}u_{s2imn}-a_{khm}u_{s3ijn}+a_{hmn}u_{s4ijk}&~
\\+b_{s1mn}v_{ijkh}-b_{s2in}v_{jkhm}+b_{s3ij}v_{khmn}=0,&\raisebox{1.5ex}[0pt]{for $1\leq i,j,k,h,m,n,s\leq2$.}
\end{array}\tag{SI(6a)}
\end{equation*}

Since $E^2=E^2_{-3}\oplus E^2_{-4}$, the relations $R=\{r_3,r_4\}$ where $\deg r_3=3$ and $\deg r_4=4$. By Lemma \ref{recover} and the $A_\infty$-algebra structure on $E$ described above, we can write
\begin{eqnarray*}
&&r_3=\sum_{1\leq i, j, k\leq2}a_{ijk}x_ix_jx_k,\\
&&r_4=\sum_{1\leq i, j, k, h\leq2}v_{ijkh}x_ix_jx_kx_h.
\end{eqnarray*}
Furthermore, $r_3$ and $r_4$ are neither zero nor a product of lower-degree polynomials since $A$ is a domain.
\vspace{1mm}
\subsection{Jordan type}~

We now concentrate on the Jordan type. We write
\begin{equation*}
\mathcal{R}=\left(
\begin{array}{cc}
-g&1\\
0&-g
\end{array}
\right).
\end{equation*}

Next, we work with $m_3$ by considering \textsf{SI(4)} to describe $r_3$. By \ref{SI(4b)}, we have
\begin{equation*}\label{SI(4c)}
\begin{array}{cr}
\multicolumn{2}{c}{(t-g^3)a_{111}=0,}\\
(t-g^3)a_{112}+g^2a_{111}=0,&(t-g^3)a_{212}+g^2(a_{211}+a_{112})-ga_{111}=0,\\
(t-g^3)a_{121}+g^2a_{111}=0,&(t-g^3)a_{221}+g^2(a_{211}+a_{121})-ga_{111}=0,\\
(t-g^3)a_{211}+g^2a_{111}=0,&(t-g^3)a_{122}+g^2(a_{112}+a_{121})-ga_{111}=0,\\
\multicolumn{2}{c}{(t-g^3)a_{222}+g^2(a_{221}+a_{212}+a_{122})-g(a_{112}+a_{121}+a_{211})+a_{111}=0.}
\end{array}\tag{SI(4c)}
\end{equation*}

If $t-g^3\neq 0$, all $a_{ijk}=0$ which implies $r_3=0$. Therefore
\begin{equation*}
t=g^3.
\end{equation*}
From (\ref{SI(4c)}), it is easy to obtain
\begin{equation*}
\left\{\begin{array}{l}a_{111}=a_{112}=a_{121}=a_{211}=0,\\
a_{221}+a_{212}+a_{122}=0.
\end{array}
\right.
\end{equation*}
Hence $r_3=a_{122}x_1x_2^2+a_{212}x_2x_1x_2+a_{221}x_2^2x_1+a_{222}x_2^3$.
Moreover, it is easy to see $a_{122}a_{221}\ne 0$ since $A$ is a domain, we write $a_{122}=1,\; a_{221}=p\neq 0 \;\; \text{and}\;\; a_{222}=w, \; a_{212}=-(1+p)$. So
$$
r_3=x_1x_2^2-(1+p)x_2x_1x_2+px_2^2x_1+wx_2^3.
$$

We get the solutions for $b_{isjk}$ from (\ref{SI(4a)}),
\begin{equation*}
\begin{array}{llll}
b_{1322}=1,      &b_{2312}=-(1+p),&b_{2321}=p        ,&b_{2322}=w,\\
b_{1222}=g(1+p),&b_{2212}=-gp    ,&b_{2221}=-g        ,&b_{2222}=1-gw,\\
b_{1122}=g^2p   ,&b_{2112}=g^2   ,&b_{2121}=-g^2(1+p),&b_{2122}=gp+g^2w,\\
\multicolumn{4}{c}{\text{the other of } b_{ijkh} \text{ are zero.} }
\end{array}
\end{equation*}

Then consider \textsf{SI(5)} to describe $r_4$. By replacing $r_4$ with the equivalent relation
$$
r_4-v_{1122}x_1r_3-v_{2122}x_2r_3-v_{1221}r_3x_1-v_{1222}r_3x_2,
$$
we may assume that
\begin{equation*}
v_{1122}=v_{2122}=v_{1221}=v_{1222}=0.
\end{equation*}

Using (\ref{SI(5a)}) recursively to eliminate $u_{isjkh}$, we obtain equations:
\begin{eqnarray}
&&(1-g^4t)v_{1111}=0,\label{SI(5a)1}\\
&&(1-g^4t)v_{1112}=-g^3tv_{1111},\label{SI(5a)2}\\
&&(1-g^4t)v_{1121}=-g^3tv_{1111},\label{SI(5a)3}\\
&&(1-g^4t)v_{1122}=-g^3t(v_{1112}+v_{1121})+g^2tv_{1111}-(g^4c_{11}+c_{21}+g^3c_{31})=0,\label{SI(5a)4}\\
&&(1-g^4t)v_{1211}=-g^3tv_{1111},\label{SI(5a)5}\\
&&(1-g^4t)v_{1212}=-g^3t(v_{1112}+v_{1211})+g^2tv_{1111}+(1+p)(g^4c_{11}+c_{21}+g^3c_{31}),\label{SI(5a)6}\\
&&\begin{aligned}
(1-g^4t)v_{1221}=&-g^3t(v_{1121}+v_{1211})+g^2tv_{1111}\\
                 &-p(g^4c_{11}+c_{21}+g^3c_{31})+gc_{11}+g^4c_{21}+c_{31}=0,
\end{aligned}\label{SI(5a)7}\\
&&\begin{aligned}
(1-g^4t)v_{1222}=&-g^3tv_{1212}+g^2t(v_{1112}+v_{1121}+v_{1211})-gtv_{1111}\\
                 &-c_{11}-g^3c_{21}-w(g^4c_{11}+c_{21}+g^3c_{31})+gc_{12}+g^4c_{22}+c_{32}=0,
\end{aligned}\label{SI(5a)8}\\
&&(1-g^4t)v_{2111}=-g^3tv_{1111},\label{SI(5a)9}\\
&&(1-g^4t)v_{2112}=-g^3t(v_{1112}+v_{2111})+g^2tv_{1111},\label{SI(5a)10}\\
&&(1-g^4t)v_{2121}=-g^3t(v_{1121}+v_{2111})+g^2tv_{1111}-(1+p)(gc_{11}+g^4c_{21}+c_{31}),\label{SI(5a)11}\\
&&\begin{aligned}
(1-g^4t)v_{2122}=&-g^3t(v_{2112}+v_{2121})+g^2t(v_{1112}+v_{1121}+v_{2111})-gtv_{1111}\\
                 &+(1+p+g^3)c_{11}+g^3(1+p)c_{21}\\
                 &-(1+p+g^3)gc_{12}-(g^4(1+p)+1)c_{22}-(g^3+p+1)c_{32}=0,
\end{aligned}\label{SI(5a)12}\\
&&(1-g^4t)v_{2211}=-g^3t(v_{1211}+v_{2111})+g^2tv_{1111}+p(gc_{11}+g^4c_{21}+c_{31}),\label{SI(5a)13}\\
&&\begin{aligned}
(1-g^4t)v_{2212}=&-g^3t(v_{1212}+v_{2112}+v_{2211})+g^2t(v_{1112}+v_{1211}+v_{2111})-gtv_{1111}\\
                 &+(-g^3(1+p)-p)c_{11}-g^3pc_{21}\\
                 &+(g^4(1+p)+gp)c_{12}+(g^4p+1+p)c_{22}+(p+g^3(1+p))c_{32},
\end{aligned}\label{SI(5a)14}\\
&&\begin{aligned}
(1-g^4t)v_{2221}=&-g^3t(v_{2121}+v_{2211})+g^2t(v_{1121}+v_{1211}+v_{2111})-gtv_{1111}\\
                 &+w(gc_{11}+g^4c_{21}+c_{31})+g^3pc_{11}\\
                 &-p(g^4c_{12}+c_{22}+g^3c_{32}),
\end{aligned}\label{SI(5a)15}\\
&&\begin{aligned}
(1-g^4t)v_{2222}=&-g^3t(v_{2212}+v_{2221})+g^2t(v_{1212}+v_{2112}+v_{2121}+v_{2211})\\
                 &-gt(v_{1112}+v_{1121}+v_{1211}+v_{2111})+tv_{1111}\\
                 &+w(g^3c_{11}-c_{11}-g^3c_{21})\\
                 &+w((-g^4+g)c_{12}+(g^4-1)c_{22}+(1-g^3)c_{32}).
\end{aligned}\label{SI(5a)16}
\end{eqnarray}

If $1-g^4t\neq 0$, then all $v_{1ijk}=0$, which implies $x_2$ is a zero divisor, contrary to our assumption. Hence
\begin{equation*}
g^4t=1, \quad \quad g^7=1.
\end{equation*}

From (\ref{SI(5a)2}), we get $v_{1111}=0$. Hence (\ref{SI(5a)4}), (\ref{SI(5a)6}), (\ref{SI(5a)7}) become
\begin{eqnarray}
&&v_{1112}+v_{1121}=-gM,\label{SI(5a)41}\\
&&v_{1112}+v_{1211}=(1+p)gM,\label{SI(5a)61}\\
&&v_{1121}+v_{1211}=(g^4-p)gM,\label{SI(5a)71}
\end{eqnarray}
where $M:=g^4c_{11}+c_{21}+g^3c_{31}$. The equations
(\ref{SI(5a)10}), (\ref{SI(5a)11}), (\ref{SI(5a)13}) become
\begin{eqnarray}
&&v_{1112}+v_{2111}=0,\label{SI(5a)101}\\
&&v_{1121}+v_{2111}=-(1+p)g^5M,\label{SI(5a)111}\\
&&v_{1211}+v_{2111}=pg^5M,\label{SI(5a)131}
\end{eqnarray}

The equations (\ref{SI(5a)41})-(\ref{SI(5a)131}) hold is same to the following equations to be held
\begin{eqnarray*}
&-gM=(g^4-p-pg^4)gM,\\
&(1+p)gM=(2g^4-p+pg^4)gM.
\end{eqnarray*}
Hence we have two cases:
\begin{description}
\item[Case 1]  $M=0$.
\item[Case 2]  $M\neq 0, g=1, p=1$.
\end{description}

\section{Regular algebras of Jordan type}
We continue to analyze the $A_\infty$-algebra structures this section. We solve all the algebras corresponding to Case 1 and Case 2, and prove that there is one class of AS-regular algebras in Case 2, and no AS-regular algebra in Case 1.

\begin{proposition}
Suppose that $A$ is an AS-regular algebra of type $(12221)$ which is Jordan type, then Case 1 gives no AS-regular algebras and Case 2 gives exactly one class of AS-regular algebras.
\end{proposition}

\subsection{Case 1: non-AS-regular algebras}

If $M=0$, the equations (\ref{SI(5a)41})-(\ref{SI(5a)131}) tell that
\begin{equation*}
v_{1112}=v_{1121}=v_{1211}=v_{2111}=0.
\end{equation*}
Then
\begin{equation*}
\begin{aligned}
r_4=&v_{1212}x_1x_2x_1x_2+v_{2112}x_2x_1^2x_2+v_{2121}x_2x_1x_2x_1+v_{2211}x_2^2x_1^2\\
    &+v_{2212}x_2^2x_1x_2+v_{2221}x_2^3x_1+v_{2222}x_2^4.
\end{aligned}
\end{equation*}
Since $A$ is a domain, $v_{1212}$ must be nonzero. Hence we assume
$v_{1212}=1$.
Now
\begin{equation*}
\begin{aligned}
r_4=&x_1x_2x_1x_2+v_{2112}x_2x_1^2x_2+v_{2121}x_2x_1x_2x_1+v_{2211}x_2^2x_1^2\\
    &+v_{2212}x_2^2x_1x_2+v_{2221}x_2^3x_1+v_{2222}x_2^4.
\end{aligned}
\end{equation*}

Next, we start to perform the computations. Find expressions of $\{v_{1212}, v_{2112}, v_{2121}, v_{2211}, v_{2212}\}$ from (\ref{SI(5a)1})-(\ref{SI(5a)16}) which are represented by $\{g,p,$ $w,c_{11}, c_{21}, c_{12}, c_{22}, c_{32}, v_{2221}\}$, and formulas of $u_{isjkh}$ from \ref{SI(5a)}. We omit them because of its length. Then input $\{a_{ijk}, b_{stmn}, u_{isjkh}, v_{ijkh}\}$ into \ref{SI(6a)}. This produces $2^7 $ equations involving the variables $\{g,p,w,c_{11},c_{21},c_{12},c_{22},c_{32},v_{2221},v_{2222}\}$. We compute those and solve the equations by Maple.

After deleting useless solutions, we have five different solutions in total. Input them into the coefficients of $r_3,r_4$ as listed below.

\begin{description}
\item[Solution 1]
\begin{equation*}\label{S1}
\begin{aligned}
&\quad\quad~~~ g=1,\quad\quad p=1,\quad\quad w=0,\\
&v_{1212}=1,~~v_{2112}= -1,~~v_{2121}=-1,~~v_{2211}= 1,\\
&\quad~ v_{2212}=c_{22},\quad v_{2221}=-c_{22},\quad v_{2222}= v_{2222}.
\end{aligned}\tag{S1}
\end{equation*}

\item[Solution 2]
\begin{equation*}\label{S2}
\begin{aligned}
&\quad\quad~~~ g=1,\quad\quad p=-1,\quad\quad w=w,\\
&~~v_{1212}=1,~~v_{2112}= 1,~~v_{2121}=1,~~v_{2211}=-3,\\
&v_{2212}=1-\frac{w}{2},\quad v_{2221}=\frac{7w}{2}-1,\quad v_{2222}=-\frac{3w^2}{2}+\frac{w}{2}.\\
\end{aligned}\tag{S2}
\end{equation*}

\item[Solution 3]
\begin{equation*}\label{S3}
\begin{aligned}
&\quad\quad~~~ g=1,\quad\quad p=-1,\quad\quad w=\frac{2}{7},\\
&v_{1212}=1,~~v_{2112}= 1,~~v_{2121}=1,~~v_{2211}=-3,\\
&\quad~ v_{2212}=\frac{6}{7},~\quad v_{2221}=0,~\quad v_{2222}= v_{2222}.
\end{aligned}\tag{S3}
\end{equation*}

\item[Solution 4]
\begin{equation*}\label{S4}
\begin{aligned}
&\quad\quad~~~ g=j ,\quad\quad p=-j ^3,\quad\quad w=w,\\
&~~v_{1212}=1,~~~v_{2112}=j,~~~v_{2121}=-j ^6-j ^2-2j -2,\\
&v_{2211}=j ^6+j ^2+j +1,v_{2212}=-w(\frac{j^4}{2}+2j ^3+3j ^2+2+\frac{7j}{2})+\frac{j ^6+1}{2},\\ &v_{2221}=w(j ^5+\frac{3j^4}{2}+2j ^3+3j ^2+\frac{7j}{2}+3)-\frac{j ^6+1}{2},\\ &v_{2222}=\frac12\big({w^2(-4j ^5+10j ^3+14j ^2+13j +6)-w(j ^3+2j ^2+2j +1)}\big).
\end{aligned}\tag{S4}
\end{equation*}

\item[Solution 5]
\begin{equation*}\label{S5}
\begin{aligned}
&\quad\quad~~~ g=j,\quad\quad p=j^2,\quad w=c_{22}(-j^6+j^5),\\
&~~v_{1212}=1,~~~v_{2112}=-1,~~~v_{2121}=-j^2,\\
&v_{2211}=j^2,v_{2212}=(j^4-j^3+j^2)c_{22},\\
&v_{2221}=c_{22}(2j^5+2j^3+j+1),\\
&v_{2222}=c_{22}^2(j^6-2j^5-j^3-j^2-2).
\end{aligned}\tag{S5}
\end{equation*}
\end{description}
where $j^6+j^5+j^4+j^3+j^2+j+1=0$.
\vspace{2mm}

We check Hilbert series of them by using Diamond Lemma \cite{B} to calculate the Gr\"obner bases. Before that, we show a lemma to help us compare the Hilbert series with other series in low degrees. Following we fix an arbitrary monomial ordering on $X^*$. For any $u,v\in X^*$, we say $v$ is a \emph{factor} of $u$, if there exist $w,w'\in X^*$ such that $u=wvw'$ denoted by $v|u$. Let any nonzero polynomial $f\in k\langle X\rangle$, the \emph{leading monomial} $LM(f)$ of $f$ is the largest monomial in $f$. Let $\mathcal{G}$ be the reduced monic Gr\"obner basis of $I$, and $\mathcal{G}=\bigcup_i\mathcal{G}_i$ where $\mathcal{G}_i=\{f\in \mathcal{G}\;|\;\deg f\leq i\}$. Then the set
$$
NW(\mathcal{G})=\{u\in X^*\; |\; LM(g)\nmid u\text{ for any }g\in \mathcal{G}\}=\bigcup_iNW(\mathcal{G})_i
$$
is a $k$-basis of $A$, where $NW(\mathcal{G})_i$ consists of the elements of degree $i$ in $NW(\mathcal{G})$. Hence $\dim_kA_m=\#(NW(\mathcal{G})_m)$. Notice that $NW(\mathcal{G})_i=\{u\in X^*\;|\;LM(g)\nmid u\text{ for  any }g \in \mathcal{G}_i\text{ and }\deg u=i\}$.

\begin{lemma}\label{low degree LM and Hilbert}
Let $A=k\langle X \rangle/I$  be a connected graded algebra, $\mathcal{G}$ is a Gr\"obner basis of $I$, and let $A'=k\langle X\rangle/(LM(\mathcal{G}_m))$. Then $H_{A'}(t)-H_A(t)=\sum_{i>m}a_it^i$ with $a_i\geq0$.
\end{lemma}
\begin{proof}
For any $i\geq 0$, $LM(\mathcal{G}_m)_i\subset LM(\mathcal{G})_i$ since $LM(\mathcal{G}_m)\subset LM(\mathcal{G})$. Then $NW(LM(\mathcal{G}))_i\subset NW(LM(\mathcal{G}_m))_i$ and $\dim_kA'_i\geq\dim_kA_i$. $H_{A'}(t)-H_A(t)=\sum_{i\geq 0}a_it^i$ with $a_i\geq 0$.

Moreover, for $i\leq m$,
\begin{equation*}
\begin{aligned}
NW(LM(\mathcal{G}_m))_i&=\{u\in X^*\;|\;LM(g)\nmid u\text{ for  any }g \in LM(\mathcal{G}_m) \text{ and }\deg u=i\}\\
&=\{u\in X^*\;|\;LM(g)\nmid u\text{ for  any }g \in \mathcal{G}_m \text{ and }\deg u=i\}\\
&=\{u\in X^*\;|\;LM(g)\nmid u\text{ for  any }g \in \mathcal{G}_i \text{ and }\deg u=i\}\\
&=NW(\mathcal{G})_i.
\end{aligned}
\end{equation*}
Hence $\dim_kA_i=\dim_kA'_i$ if $i\leq m$, which implies $H_{A'}(t)-H_A(t)=\sum_{i>m}a_it^i$ with $a_i\geq0$.
\end{proof}

We choose monomial ordering $\prec_{gr-lex}$ on the free monoid $\{x_1,x_2\}^*$ as follows: For any $u=x_{i_1}x_{i_2}\cdots x_{i_s}, v=x_{j_1}x_{j_2}\cdots x_{j_t}\in X^*$, to say $u\prec_{gr-lex} v$ we mean either
\begin{enumerate}
\item $s<t$, or
\item $s=t$ and there exists $p$ such that $x_{i_l}=x_{j_l}$ for $l<p$ and $i_p>j_p$.
\end{enumerate}

Keep in mind that  Hilbert series of type $(12221)$ is
\begin{equation*}\label{Hilbert series}
H_A(t)=1+2t+4t^2+7t^3+11t^4+16t^5+23t^6+31t^7+\cdots.\tag{HS}
\end{equation*}

The algebra corresponding to (\ref{S1}) is  $U(g,h)=k\langle x_1,x_2\rangle/(f_{11},f_{12})$, where
\begin{equation*}
\begin{aligned}
&f_{11}=x_1x_2^2-2x_2x_1x_2+x_2^2x_1,\\
&f_{12}=x_1x_2x_1x_2-x_2x_1^2x_2-x_2x_1x_2x_1
+x_2^2x_1^2+gx_2^2x_1x_2-gx_2^3x_1+hx_2^4,
\end{aligned}
\end{equation*}
with $g, h\in k$.
\vspace{2mm}

$V(w,l)=k\langle x_1,x_2\rangle/(f_{21},f_{22})$ is the algebra corresponding to (\ref{S5}), where
\begin{equation*}
\begin{aligned}
f_{21}=&x_1x_2^2-(1+j^2)x_2x_1x_2+j^2x_2^2x_1+w(-j^6+j^5)x_2^3,\\
f_{22}=&x_1x_2x_1x_2-x_2x_1^2x_2-j^2x_2x_1x_2x_1+j^2x_2^2x_1^2+l(j^4-j^3+j^2)x_2^2x_1x_2\\
&+l(2j^5+2j^3+j+1)x_2^3x_1+l^2(j^6-2j^5-j^3-j^2-2)x_2^4,
\end{aligned}
\end{equation*}
with $j^6+j^5+j^4+j^3+j^2+j+1=0$ and $w,l\in k$.

\begin{lemma}
$U(g,h),V(w,l)$ are not AS-regular.
\end{lemma}
\begin{proof}
By Diamond Lemma, we know $\{f_{11},f_{12}\}$ and $\{f_{21},f_{22}\}$ are Gr\"obner bases of $(f_{11},f_{12})$ and $(f_{21},f_{22})$ respectively. Then their leading monomials of Gr\"obner bases are the same, that is, $LM(\mathcal{G})=\{x_1x_2^2,x_1x_2x_1x_2\}$. Denote $MON:=k\langle x_1,x_2\rangle/(LM(\mathcal{G}))$.Hence
$$
H_{U(g, h)}(t)=H_{V(w, l)}(t)=H_{MON}(t)=1+2t+4t^2+7t^3+11t^4+17t^5+\cdots,
$$
which is different from (\ref{Hilbert series}).
\end{proof}

The algebras corresponding to solutions (\ref{S2})---(\ref{S4}) are listed below.
\begin{description}
\item[(S2)]
$O(w)=k\langle x_1,x_2\rangle/(f_{31},f_{32})$, where
\begin{equation*}
\begin{aligned}
f_{31}=&x_1x_2^2-x_2^2x_1+wx_2^3,\\
f_{32}=&x_1x_2x_1x_2+x_2x_1^2x_2+x_2x_1x_2x_1-3x_2^2x_1^2+(1-\frac{w}{2})x_2^2x_1x_2\\
&+(\frac{7w}{2}-1)x_2^3x_1+(-\frac{3w^2}{2}+\frac{w}{2})x_2^4,
\end{aligned}
\end{equation*}
with $w\in k$.
\vspace{2mm}
\item[(S3)]
$P(a)=k\langle x_1,x_2\rangle/(f_{41},f_{42})$, where
\begin{equation*}
\begin{aligned}
f_{41}=&x_1x_2^2-x_2^2x_1+\frac{2}{7}x_2^3,\\
f_{42}=&x_1x_2x_1x_2+x_2x_1^2x_2+x_2x_1x_2x_1-3x_2^2x_1^2+\frac{6}{7}x_2^2x_1x_2+ax_2^4,
\end{aligned}
\end{equation*}
with $a\in k$.
\vspace{2mm}
\item[(S4)]
$Q(d)=k\langle x_1,x_2\rangle/(f_{51},f_{52})$, where
\begin{equation*}
\begin{aligned}
f_{51}=&x_1x_2^2-(1-j^3)x_2x_1x_2+j^3x_2^2x_1+dx_2^3,\\
f_{52}=&x_1x_2x_1x_2+jx_2x_1^2x_2-(j ^6+j ^2+2j +2)x_2x_1x_2x_1+(j ^6+j ^2+j +1)x_2^2x_1^2\\
&+\big(\frac{j ^6+1}{2}-d(\frac{j^4}{2}+2j ^3+3j ^2+2+\frac{7j}{2})\big)x_2^2x_1x_2\\
&+\big(d(j ^5+\frac{3j^4}{2}+2j ^3+3j ^2+\frac{7j}{2}+3)-\frac{j ^6+1}{2}\big)x_2^3x_1\\
&+\frac12\Big({d^2(-4j ^5+10j ^3+14j ^2+13j +6)-d(j ^3+2j ^2+2j +1)}\Big)x_2^4,
\end{aligned}
\end{equation*}
with $j^6+j^5+j^4+j^3+j^2+j+1=0$ and $d\in k$.
\end{description}

\begin{lemma}
$O(w),P(a),Q(d)$ are not AS-regular.
\end{lemma}
\begin{proof}
 Only consider $\mathcal{G}_7$. Then we obtain that the leading monomials in $\mathcal{G}_7$ of them are the same, that is,
$$
LM(\mathcal{G}_7)=\{x_1x_2^2,x_1x_2x_1x_2,x_2^2x_1^2x_2,x_2^2x_1^3x_2,x_2^2x_1x_2x_1^2x_2,x_2^2x_1^4x_2\}.
$$
Let $MON_7:=k\langle x_1,x_2\rangle/(\mathcal{G}_7)$.

Suppose they are AS-regular, then they have the same Hilbert series (\ref{Hilbert series}) denoted $H(t)$. Then $H_{MON_7}(t)-H(t)=\sum_{i\geq 8}a_it^i$ with $a_i\geq 0$ by Lemma \ref{low degree LM and Hilbert}.

However, $H_{MON_7}(t)=1+2t+4t^2+7t^3+11t^4+16t^5+23t^6+32t^7+\cdots$
and
$$H_{MON_7}(t)-H(t)=t^7+\Sigma_{i\geq 8}a_it^i.$$
It is a contradiction.
\end{proof}
\vspace{2mm}

\subsection{Case 2: AS-regular}~

Now we turn to Case 2. From (\ref{SI(5a)}), we find that $v_{1112}=-\frac{M}{2}$. Assume
$$
v_{1112}=1.
$$

The same method is used as in Case 1. Using (\ref{SI(5a)1})--(\ref{SI(5a)16}) again, represent $v_{ijkh}$ by $\{w,c_{11},c_{21},$ $c_{31},c_{12},c_{22},c_{32},v_{2221},v_{2222}\}$. Find expressions for $u_{isjkh}$ from \ref{SI(5a)}. We also omit those explicit formulas. Then input the formulas into \ref{SI(6a)} which produces $2^7 $ equations involving the variables $\{w,c_{11},c_{21},c_{31},c_{12},c_{22},c_{32},v_{2221},v_{2222}\}$ and solve them. All those steps are computed by Maple. There exits only one solution, and take it back to $v_{ijkh}$:
\begin{description}
\item[Solution 6]
\begin{equation*}\label{S6}
\begin{aligned}
&\quad\quad~~~ g=1,\quad\quad p=1,\quad\quad w=0,\\
&v_{1112}=1,~~v_{1121}=-3,~~v_{1211}=3,\\
&v_{1212}=-c_{21}+1,~~ v_{2111}=-1,\quad v_{2112}=c_{21},\\
&v_{2121}=c_{21}-3,~v_{2211} = -c_{21}+2, ~v_{2212}=-v_{2221},\\
&v_{2221}= v_{2221}, ~~~v_{2222}= v_{2222}.
\end{aligned}\tag{S6}
\end{equation*}
\end{description}

The corresponding algebra is
$\mathcal{J}=k\langle x_1,x_2\rangle/(f_1,f_2)$ where
\begin{equation*}
\begin{aligned}
f_1=&x_1x_2^2-2x_2x_1x_2+x_2^2x_1,\\
f_2=&x_1^3x_2-3x_1^2x_2x_1+3x_1x_2x_1^2-x_2x_1^3+(1-u)x_1x_2x_1x_2+ux_2x_1^2x_2\\
&+(u-3)x_2x_1x_2x_1+(2-u)x_2^2x_1^2-vx_2^2x_1x_2+vx_2^3x_1+wx_2^4,
\end{aligned}
\end{equation*}
and $u, v, w\in k$.

\vspace{2mm}
Then we define $\mathbb{Z}^2$-grading on $k\langle x_1, x_2\rangle$ with $\deg^2 x_1=(1,0),\; \deg^2 x_2=(0,1)$. We choose $\prec_{gr-lex}$ on $\{x_1,x_2\}^*$ as the monomial ordering defined in Section 4.1. The admissible ordering $\prec_{\mathbb{Z}^2}$ is defined as in Section \ref{Z^n}. Let $I=(f_1, f_2)$ and $\mathcal{G}$ be the Gr\"obner basis of $I$ respect to $\prec_{\mathbb{Z}^2}$. Then
$$
G^2(\mathcal{J})\cong k\langle x_1,x_2\rangle/(LH(I))\cong k\langle x_1,x_2\rangle/(LH(\mathcal{G})).
$$
Applying the Diamond Lemma, the Gr\"obner basis $\mathcal{G}$ is $\{f_1, f_2, f_3\}$ where
\begin{equation*}
\begin{aligned}
f_3=&x_1^2x_2x_1x_2-3x_1x_2x_1^2x_2+2x_1x_2x_1x_2x_1+3x_2x_1^2x_2x_1-5x_2x_1x_2x_1^2\\
&+(2u-2)x_2x_1x_2x_1x_2+2x_2^2x_1^3-2ux_2^2x_1^2x_2+(6-2u)x_2^2x_1x_2x_1\\
&+(2u-4)x_2^3x_1^2+2vx_2^3x_1x_2-2vx_2^4x_1-2wx_2^5.
\end{aligned}
\end{equation*}
So, $LH(\mathcal{G})=\{LH(f_1), LH(f_2), LH(f_3)\}$,
where
\begin{equation*}
\begin{aligned}
&LH(f_1)=x_1x_2^2-2x_2x_1x_2+x_2^2x_1,\\
&LH(f_2)=x_1^3x_2-3x_1^2x_2x_1+3x_1x_2x_1^2-x_2x_1^3,\\
&LH(f_3)=x_1^2x_2x_1x_2-3x_1x_2x_1^2x_2+2x_1x_2x_1x_2x_1+3x_2x_1^2x_2x_1-5x_2x_1x_2x_1^2+2x_2^2x_1^3.
\end{aligned}
\end{equation*}
However, $LH(f_3)=LH(f_2)x_2-x_2LH(f_2)+LH(f_1)x_1^2-x_1^2LH(f_1)-x_1LH(f_1)x_1$. Therefore,
$$
G^2(\mathcal{J})=k\langle x_1,x_2\rangle/(x_1x_2^2-2x_2x_1x_2+x_2^2x_1, x_1^3x_2-3x_1^2x_2x_1+3x_1x_2x_1^2-x_2x_1^3).
$$
This is just $D(-2,-1)$ in \cite{LPWZ2} which is AS-regular. By Theorem \ref{first theo}, we have

\begin{proposition}
$\mathcal{J}$ is an AS-regular algebra of global dimension $4$.
\end{proposition}

\section{Properties of the algebras}
In this section, we show some properties of $\mathcal{J}$ about ring-theoretic, homology and geometry.

$D(-2,-1)$ has been proved to be noetherian, strongly noetherian and Auslander regular in \cite{LPWZ2}.  By Corollary \ref{nsnac1}, we obtain immediately
\begin{theorem}
The algebra $\mathcal{J}$ is strongly noetherian and Auslander regular.
\end{theorem}

Besides, we still want to know whether $\mathcal{J}$ is Cohen-Macaulay. An Ore extension is constructed below.
\begin{theorem}
The algebra $\mathcal{J}$ is Cohen-Macaulay.
\end{theorem}
\begin{proof}
We claim $\mathcal{J}$ is an Ore extension of an algebra which is Cohen-Macaulay. Hence $\mathcal{J}$ is Cohen-Macaulay by \cite[Lemma 1.3]{RZ}.

Take a graded polynomial algebra $B=k[x_2, z_1, z_2]$ with $\deg x_2=1,\deg z_1=2,$ and $\deg z_2=3$. This is Cohen-Macaulay since it is an iterated Ore extension.

Let $x_1$ be a new variable with degree 1 and $C=B[x_1;\sigma,\delta]$ where $\sigma$ is identity and
\begin{equation*}
\delta(x_2)=z_1,\quad\delta(z_1)=z_2,\quad\delta(z_2)=(u-1)z_1^2-x_2z_2
+vx_2^2z_1-wx_2^4,\quad u, v, w\in k.
\end{equation*}

We rewrite the relations between $x_1$ and $x_2,z_1$ as
\begin{equation*}
x_1x_2=x_2x_1+z_1,\quad x_1z_1=z_1x_1+z_2.
\end{equation*}
Then $z_1,z_2$ can be generated by $x_1,x_2$ as
\begin{equation*}
z_1=x_1x_2-x_2x_1,\quad z_2=x_1z_1-z_1x_1.
\end{equation*}
Hence $C$ is generated by $x_1,x_2$. The other four relations of $C$ are listed below
\begin{eqnarray*}
&&x_2z_1-z_1x_2,\\
&&x_1z_2-z_2x_1+(1-u)z_1^2+x_2z_2-vx_2^2z_1+wx_2^4,\\
&&x_2z_2-z_2x_2,\\
&&z_1z_2-z_2z_1.
\end{eqnarray*}
Replacing $z_1,z_2$, the first relation is equivalent to the relation $f_1$ of $\mathcal{J}$. After being reduced by $f_1$ with respect to $\prec_{\mathbb{Z}^2}$, the second is equivalent to $f_2$ of $\mathcal{J}$. And the last two relations can be derived from $f_1,f_2,f_3$. Hence $\mathcal{J}\cong C$.
\end{proof}

\begin{remark}
The proof also shows $\mathcal{J}$ is AS-regular of dimension $4$, strongly noetherian and Auslander regular. However, finding an Ore extension is a tedious task, the method used in last section is more effective.
\end{remark}

\begin{theorem}
The automorphism group of $\mathcal{J}$ is isomorphic to the group $G$,
where
$$G=\left\{
\left(
\begin{array}{cc}
 a&b\\
0&a
\end{array}
\right)
\;\Big|\;\;a\in k\backslash\{0\},b\in k
\right\}.$$
\end{theorem}
\begin{proof}
Let $\sigma$ is an arbitrary automorphism of $\mathcal{J}$. Suppose that
$$
\sigma(x_1)=a_1x_1+a_2x_2,\quad \sigma(x_2)=b_1x_1+b_2x_2,
$$
and the matrix
$\left(
\begin{array}{cc}
a_1&a_2\\
b_1&b_2
\end{array}
\right)$
is nonsingular, that is, $a_1b_2-a_2b_1\neq 0$. Then
\begin{equation*}
\sigma(f_1)=b_1(a_2b_1-a_1b_2)(x_1^2x_2-2x_1x_2x_1+x_2x_1^2)+b_2(a_1b_2-a_2b_1)(x_1x_2^2-2x_2x_1x_2+x_2^2x_1).
\end{equation*}

It is zero in $\mathcal{J}$, so it must be a scalar multiple of $f_1$. Hence
$$
b_1=0.
$$

To see the other relation $f_2$,
\begin{equation*}
\begin{aligned}
\sigma(f_2)&=a_1^3b_2(x_1^3x_2-3x_1^2x_2x_1+3x_1x_2x_1^2-x_2x_1^3)+a_1^2b_2(4a_2+(1-u)b_2)x_1x_2x_1x_2\\
           &\quad+a_1^2b_2((u-3)b_2-4a_2)x_2x_1x_2x_1+a_1^2b_2(2a_2+(2-u)b_2)x_2^2x_1^2-2a_1^2a_2b_2x_1^2x_2^2\\
           &\quad+ua_1^2b_2^2x_2x_1^2x_2+a_1a_2b_2(a_2+(1-u)b_2)x_1x_2^3+a_1a_2b_2((2u-3)b_2-3a_2)x_2x_1x_2^2\\    &\quad+a_1b_2((3-u)a_2b_2+3a_2^2-vb_2^2)x_2^2x_1x_2+a_1b_2(vb_2^2-a_2^2-a_2b_2)x_2^3x_1+wb_2^4x_2^4\\
           &=(1-u)a_1^2b_2(b_2-a_1)x_1x_2x_1x_2+(u-3)a_1^2b_2(b_2-a_1)x_2x_1x_2x_1\\
           &\quad+(2-u)a_1^2b_2(b_2-a_1)x_2^2x_1^2+ua_1^2b_2(b_2-a_1)x_2x_1^2x_2\\
           &\quad+va_1b_2(a_1^2-b_2^2)x_2^2x_1x_2+va_1b_2(b_2^2-a_1^2)x_2^3x_1+wb_2(b_2^3-a_1^3)x_2^4.
\end{aligned}
\end{equation*}
The equalities hold in $\mathcal{J}$ and the second equality follows form $f_1, f_2$. Because the leading monomial $x_1x_2x_1x_2$ of right hand has no factors in $LM(f_1), LM(f_2)$, it must be zero. While $a_1, b_2\neq 0$, we obtain
$$
a_1=b_2.
$$

Therefore, $Aut(\mathcal{J})\cong G$.
\end{proof}

At last, we calculate the point modules of $\mathcal{J}$. Before that recall the definition.

\begin{definition} $($\cite{ATV1}$)$ Let $A$ be a connected graded algebra, a graded $A$-module $M$ is called a \emph{point module} if it satisfies the following conditions:
\begin{enumerate}
\item $M$ is generated in degree zero,
\item $M_0=k$,
\item $\dim_k M_i=1$, for all $i\geq 0$.
\end{enumerate}
\end{definition}

\begin{theorem}
$\mathcal{J}$ has two classes of point modules up to isomorphism.
\end{theorem}
\begin{proof}
Let $M=\mathcal{J}e_0$ be a point module of $\mathcal{J}$. As vector space, $M=\bigoplus^\infty_{i=0}ke_i$ where $\deg e_i=i$. The $\mathcal{J}$-module structure on $M$ can be described  by generators as
$$
x_1e_i=p_{i+1}e_{i+1},\quad x_2e_i=q_{i+1}e_{i+1},\text{ for any }i\geq0,
$$
where $p_i, q_i\in k$. For every $i>0$, $p_i, q_i$ cannot be zero simultaneously. Denote $\alpha_i=(p_i, q_i)$, then $M$ determines a unique sequence of points $\{\alpha_i\}_{i=1}^\infty$ in $\mathbb{P}^1$.

Because of the $\mathcal{J}$-module structure on $M$, we have the equations
\begin{equation*}
\begin{aligned}
&p_{i+3}q_{i+2}q_{i+1}-2q_{i+3}p_{i+2}q_{i+1}+q_{i+3}q_{i+2}p_{i+1}=0,\\
&p_{i+4}p_{i+3}p_{i+2}q_{i+1}-3p_{i+4}p_{i+3}q_{i+2}p_{i+1}+3p_{i+4}q_{i+3}p_{i+2}p_{i+1}+(1-u)p_{i+4}q_{i+3}p_{i+2}q_{i+1}\\
&\quad -q_{i+4}p_{i+3}p_{i+2}p_{i+1}+uq_{i+4}p_{i+3}p_{i+2}q_{i+1}+(u-3)q_{i+4}p_{i+3}q_{i+2}p_{i+1}+(2-u)q_{i+4}q_{i+3}p_{i+2}p_{i+1}\\
&\quad -vq_{i+4}q_{i+3}p_{i+2}q_{i+1}+vq_{i+4}q_{i+3}q_{i+2}p_{i+1}+wq_{i+4}q_{i+3}q_{i+2}q_{i+1}=0.
\end{aligned}
\end{equation*}
for any $i\geq 0$.

Notice that, the solutions of equations above are sequences of points $\{\alpha_i\}_{i=1}^\infty$ in $\mathbb{P}^1$. And those sequences of points also determines point modules of $\mathcal{J}$.

Suppose $S=\{(a_1,b_1),(a_2,b_2),\cdots\}$ is a sequence of points related to $\mathcal{J}$, it must be a solution of the equations. Now we consider the sequence $S_1=\{(a_2,b_2),(a_3,b_3),\cdots\}$, it is also a solution of the equations. It always holds for $S_i=\{(a_{i+1},b_{i+1}),(a_{i+2},b_{i+2}),\cdots\}$ for any $i>0$. In this sense, the minimum period length of those equations is 4. Hence we solve the equation as follows,

\begin{equation*}\label{equation of point modules}
\begin{aligned}
&p_{3}q_{2}q_{1}-2q_{3}p_{2}q_{1}+q_{3}q_{2}p_{1}=0,\\
&p_{4}q_{3}q_{2}-2q_{4}p_{3}q_{2}+q_{4}q_3p_{2}=0,\\
&p_{4}p_{3}p_{2}q_{1}-3p_{4}p_{3}q_{2}p_{1}+3p_{4}q_{3}p_{2}p_1+(1-u)p_{4}q_{3}p_{2}q_{1}-q_{4}p_{3}p_{2}p_{1}+uq_{4}p_{3}p_{2}q_{1}\\
&\quad +(u-3)q_{4}p_{3}q_{2}p_{1}+(2-u)q_{4}q_{3}p_{2}p_{1}
-vq_{4}q_{3}p_{2}q_{1}+vq_{4}q_3q_2p_{1}+wq_{4}q_{3}q_{2}q_{1}=0.
\end{aligned}\tag{EP}
\end{equation*}

If $p_i=0$ (respectively, $q_i=0$), then we assume $q_i=1$ (respectively, $p_i=1$) by some appropriate change of basis. If both $q_i$ and $p_i$ are nonzero, we assume $q_i=1$. Then the  solutions of equations (\ref{equation of point modules}) are listed below
\begin{equation*}\label{P1}
\left\{
\begin{array}{llll}
p_1=1,&p_2=1,&p_3=1,&p_4=1,\\
q_1=0,&q_2=0,&q_3=0,&q_4=0.
\end{array}\tag{P1}
\right.\hskip27mm
\end{equation*}
\begin{equation*}\label{P2}
\left\{
\begin{array}{llll}
p_1=p_1,&p_2=1,&p_3=1,&p_4=p_1-u,\\
q_1=1,&q_2=0,&q_3=0,&q_4=1.
\end{array}\tag{P2}
\right.\hskip18mm
\end{equation*}
\begin{equation*}\label{P3}
\left\{
\begin{array}{llll}
p_1=p_1,&p_2=p_1+d,&p_3=p_1+2d,&p_4=p_1+3d,\\
q_1=1,&q_2=1,&q_3=1,&q_4=1.
\end{array}\tag{P3}
\right.
\end{equation*}
where $d\in k$ satisfies $6d^3+(3-u)d^2-vd+w=0$.

Assemble them under the rule that each $S_i$ is a solution for any $i>0$. Therefore, there exist two classes of point modules:
\begin{enumerate}
\item (\ref{P1})(\ref{P1})(\ref{P1})(\ref{P1})(\ref{P1})(\ref{P1})$\cdots\cdots$,
\item (\ref{P3})(\ref{P3})(\ref{P3})(\ref{P3})(\ref{P3})(\ref{P3})$\cdots\cdots$.
\end{enumerate}
where (a) and (b) are sequenced by (\ref{P1}) and  (\ref{P3}) respectively.
\end{proof}

\vskip7mm

{\it Acknowledgments.}  This research is supported by the NSFC
(Grant No. 11271319).

\end{document}